\documentclass[10pt]{amsart}
\usepackage{amsmath,amsfonts,amssymb,amsthm}
\usepackage{graphics}
\usepackage{epsfig,psfrag}

\newtheoremstyle{thm}
  {9pt}{9pt}{\itshape}{}{\bfseries}{}{.5em}{}
\theoremstyle{thm}
\newtheorem{thm}{Theorem}
\newtheorem{cor}[thm]{Corollary}
\newtheorem{lemma}[thm]{Lemma}
\newtheorem{prop}[thm]{Proposition}
\newtheorem{conj}[thm]{Conjecture}

\newtheoremstyle{defin}
  {9pt}{9pt}{}{}{\bfseries}{}{.5em}{}
\theoremstyle{defin}

\newtheoremstyle{exm}
  {9pt}{9pt}{}{}{\scshape}{}{.5em}{}
\theoremstyle{exm}
\newtheorem*{exm}{Example}

\newtheoremstyle{proof}
  {}{}{}{}{\itshape}{:}{.5em}{}
\theoremstyle{proof}

\newcommand{\set}[1]{\{#1\}}

\newcommand{\N}{{\mathbb N}}

\renewcommand{\wp}{{\widetilde p}}
\newcommand{\bp}{{\bar p}}

\DeclareMathOperator{\syt}{SYT}
\DeclareMathOperator{\rib}{rib}
\DeclareMathOperator{\wt}{wt}
\DeclareMathOperator{\hgt}{ht}
\DeclareMathOperator{\jdt}{jdt}
\DeclareMathOperator{\hs}{hs}
\DeclareMathOperator{\vs}{vs}
\DeclareMathOperator{\br}{br}
\DeclareMathOperator{\sk}{sk}

\title{Skew quantum Murnaghan-Nakayama rule}
\author{Matja\v z Konvalinka}

\begin{document}

\begin{abstract}
 In this paper, we extend recent results of Assaf and McNamara on skew Pieri rule and skew Murnaghan-Nakayama rule to a more general identity, which gives an elegant expansion of the product of a skew Schur function with a quantum power sum function in terms of skew Schur functions. We give two proofs, one completely bijective in the spirit of Assaf-McNamara's original proof, and one via Lam-Lauve-Sotille's skew Littlewood-Richardson rule. We end with some conjectures for skew rules for Hall-Littlewood polynomials.
\end{abstract}

\maketitle

\section{Introduction}

Let us start with some basic definitions. A \emph{partition} $\lambda$ of $n$ is a sequence $(\lambda_1,\lambda_2,\ldots,\lambda_\ell)$ satisfying $\lambda_1 \geq \lambda_2 \geq \ldots \geq \lambda_\ell > 0$ and $\lambda_1+\lambda_2+\ldots+\lambda_\ell = n$; we use the notation $\lambda \vdash n$, $k = \ell(\lambda)$ (\emph{length} of $\lambda$), $n = |\lambda|$ (\emph{size} of $\lambda$), $\lambda_i = 0$ if $i > \ell(\lambda)$. We sometimes write $(\lambda_1^{k_1},\lambda_2^{k_2},\ldots)$ if $\lambda_1$ is repeated $k_1$ times, $\lambda_2 < \lambda_1$ is repeated $k_2$ times etc. The \emph{conjugate} partition of $\lambda$, denoted $\lambda^c$, is the partition $\mu = (\mu_1,\mu_2,\ldots,\mu_{\lambda_1})$ defined by $\mu_i = \max\set{j \colon \lambda_j \geq i}$. The \emph{Young diagram} $[\lambda]$ of a partition $\lambda$ is the set $\set{(i,j) \colon 1 \leq i \leq \ell(\lambda), 1 \leq j \leq \lambda_i}$. For partitions $\lambda,\mu$ we say that $\mu \subseteq \lambda$ if $\mu_i \leq \lambda_i$ for all $i$. If $\mu \subseteq \lambda$, the \emph{skew Young diagram} $[\lambda/\mu]$ of $\lambda/\mu$ is the set $\set{(i,j) \colon 1 \leq i \leq \ell(\lambda), \mu_i < j \leq \lambda_i}$. We denote $|\lambda| - |\mu|$ by $|\lambda/\mu|$. The elements of $[\lambda/\mu]$ are called \emph{cells}. We treat $\lambda$ and $\lambda/\emptyset$ as identical.

\medskip

We say that $\lambda/\mu$ is a \emph{horizontal strip} (respectively \emph{vertical strip}) if $[\lambda/\mu]$ contains no $2 \times 1$ (respectively $1 \times 2$) block, equivalently, if $\lambda_i^c \leq \mu_i^c + 1$ (respectively $\lambda_i \leq \mu_i+1$) for all $i$. We say that $\lambda/\mu$ is a \emph{ribbon} if $[\lambda/\mu]$ is connected and if it contains no $2 \times 2$ block, and that $\lambda/\mu$ is a \emph{broken ribbon} if $[\lambda/\mu]$ contains no $2 \times 2$ block, equivalently, if $\lambda_i \leq \mu_{i-1}+1$ for $i \geq 2$. The Young diagram of a broken ribbon is a disjoint union of $\rib(\lambda/\mu)$ number of ribbons. The \emph{height} $\hgt(\lambda/\mu)$ (respectively \emph{width} $\wt(\lambda/\mu)$) of a ribbon is the number of non-empty rows (respectively columns) of $[\lambda/\mu]$, minus $1$. The height (respectively width) of a broken ribbon is the sum of heights (respectively widths) of the components. Clearly, $\lambda/\mu$ is a horizontal (respectively vertical) strip if and only if it is a broken ribbon of height (respectively width) $0$. Figure \ref{fig1} shows examples of a horizontal strip, vertical strip, ribbon (with $\hgt(\lambda/\mu) = 8$ and $\wt(\lambda/\mu) = 7$) and broken ribbon (with $\hgt(\lambda/\mu) = 6$, $\wt(\lambda/\mu) = 6$ and $\rib(\lambda/\mu) = 3$).

\medskip

\begin{figure}[ht!]
\begin{center}
 \includegraphics[width = 0.5\textwidth]{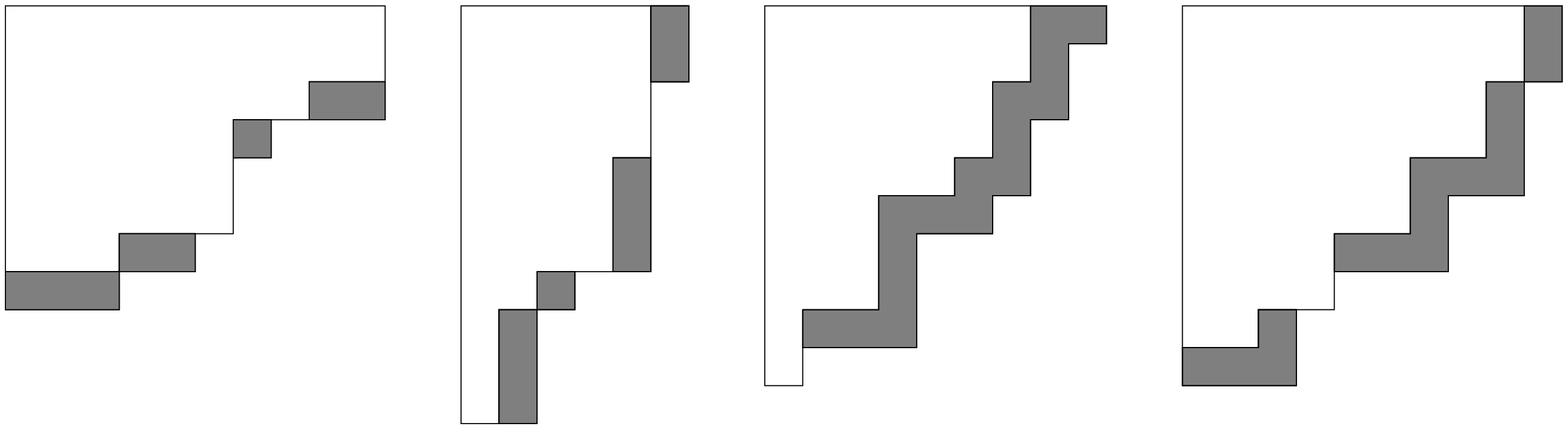}
\end{center}
\caption{}\label{fig1}\end{figure}

A map $T \colon [\lambda/\mu] \to \N$ is called a \emph{skew semistandard Young tableau of shape $\lambda/\mu$} if $T(i,j_1) \leq T(i,j_2)$ for $j_1 < j_2$, and $T(i_1,j) < T(i_2,j)$ for $i_1 < i_2$. If $T$ is a skew semistandard Young tableau, we denote by $t_i(T)$ the number of cells that map to $i$. Define the \emph{skew Schur function}
$$s_{\lambda/\mu} = \sum_T x_1^{t_1(T)} x_2^{t_2(T)} \cdots,$$
where the sum is over all semistandard Young tableaux of shape $\lambda/\mu$. A skew Schur function is a formal power series in $x_1,x_2,\ldots$, and it is easy to see that it is a symmetric function. Moreover, the set of \emph{Schur functions} $\set{s_\lambda \colon \lambda \mbox{ partition}}$ is a basis of the space of symmetric functions. For more details, and for some of the amazing properties of Schur functions, see \cite[\S 7]{stanley}.

\medskip

There are several other bases of the space of symmetric functions. For the purposes of this paper, the most important one is the \emph{power sum basis} $\set{p_\lambda \colon \lambda \mbox{ partition}}$, defined by
$$p_r = x_1^r + x_2^r + \ldots,$$
$$p_\lambda = p_{\lambda_1} p_{\lambda_2} \cdots p_{\lambda_\ell}.$$
Let us also mention the \emph{monomial basis} $\set{m_\lambda \colon \lambda \mbox{ partition}}$, defined by
$$m_\lambda = \sum x_{\pi(1)}^{\lambda_1} \cdots x_{\pi(\ell)}^{\lambda_\ell},$$
where the sum is over all injective maps $\pi \colon \set{1,\ldots,\ell} \to \N$.

\medskip

The product of Schur functions can be (uniquely) expressed as a linear combination of Schur functions:
$$s_\lambda s_\mu = \sum c_{\lambda,\mu}^\nu s_\nu.$$
The coefficients $c_{\lambda,\mu}^\nu$ are called \emph{Littlewood-Richardson coefficients} and can be computed using the celebrated \emph{Littlewood-Richardson rule}, see \cite[Appendix A1.3]{stanley}. This rule is quite complicated, but it is very simple if $\mu$ has only one row or column. Namely, we have the \emph{Pieri rule}:
\begin{equation} \label{pieri}
 s_\lambda s_r = \sum s_{\lambda^+},
\end{equation}
where the sum on the right is over all $\lambda^+$ such that $\lambda^+/\lambda$ is a horizontal strip of size $r$. Similarly, the \emph{conjugate Pieri rule} says that
\begin{equation} \label{conjpieri}
 s_\lambda s_{1^r} = \sum s_{\lambda^+},
\end{equation}
where the sum on the right is over all $\lambda^+$ such that $\lambda^+/\lambda$ is a vertical strip of size $r$.

\medskip

We also have a rule for the product of a Schur function with a power sum symmetric function, the \emph{Murnaghan-Nakayama rule}:
\begin{equation} \label{mnr}
 s_{\lambda} p_r = \sum (-1)^{\hgt(\lambda^+/\lambda)} s_{\lambda^+},
\end{equation}
where the sum on the right is over all $\lambda^+$ such that $\lambda^+/\lambda$ is a ribbon of size $r$. See \cite[Theorem 7.15.7]{stanley}.

\medskip

In \cite{amcn} and \cite{slides}, Assaf and McNamara found a beautiful extension of both Pieri rule and Murnaghan-Nakayama rule.

\begin{thm}[Skew Pieri Rule -- SPR]
 For any partitions $\lambda,\mu$, $\mu \subseteq \lambda$, we have
 $$s_{\lambda/\mu} \cdot s_r = \sum_j (-1)^j \sum s_{\lambda^+/\mu^-},$$
 where the inner sum on the right is over all $\lambda^+,\mu^-$ such that $\lambda^+/\lambda$ is a horizontal strip of size $r-j$, and $\mu/\mu^-$ is a vertical strip of size $j$.
\end{thm}

The skew Pieri rule has a dual, conjugate equivalent.

\begin{cor}[Conjugate skew Pieri rule -- CSPR]
 For any partitions $\lambda,\mu$, $\mu \subseteq \lambda$, we have
 $$s_{\lambda/\mu} \cdot s_{1^r} = \sum_j (-1)^j \sum s_{\lambda^+/\mu^-},$$
 where the inner sum on the right is over all $\lambda^+,\mu^-$ such that $\lambda^+/\lambda$ is a vertical strip of size $r-j$, and $\mu/\mu^-$ is a horizontal strip of size $j$
\end{cor}

CSPR can be proved from SPR via the involution $\omega$ on the algebra of symmetric functions, which maps $s_{\lambda/\mu}$ to $s_{\lambda^c/\mu^c}$ and preserves the product. See \cite[\S 7.6 and \S 7.14]{stanley} for details.

\begin{thm}[Skew Murnaghan-Nakayama Rule -- SMNR]
 For any partitions $\lambda,\mu$, $\mu \subseteq \lambda$, we have
 $$s_{\lambda/\mu} \cdot p_r = \sum (-1)^{\hgt(\lambda^+/\lambda)}s_{\lambda^+/\mu} - \sum (-1)^{\hgt(\mu/\mu^-)} s_{\lambda/\mu^-},$$
 where the first (respectively second) sum on the right is over all $\lambda^+$ (respectively $\mu^-$) such that $\lambda^+/\lambda$ (respectively $\mu/\mu^-$) is a ribbon of size $r$.
\end{thm}

\begin{exm}
 By SPR, we have
 $$s_{322/11} \cdot s_2 = s_{522/11} + s_{432/11} + s_{4221/11} + s_{3321/11} + s_{3222/11} - s_{422/1} - s_{332/1} - s_{3221/1} + s_{322},$$
 as shown by Figure \ref{fig12}.

 \begin{figure}[ht!]
 \begin{center}
  \includegraphics[height = 1.6cm]{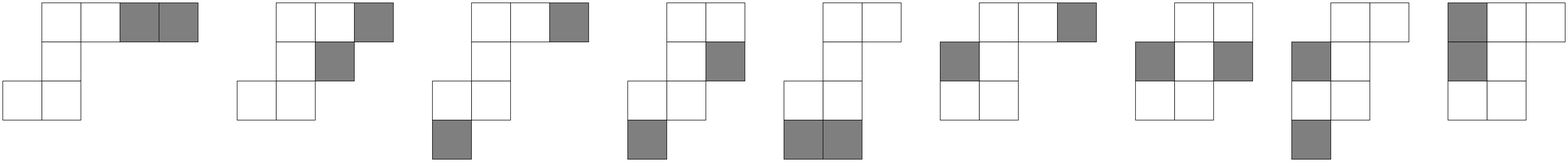}
 \end{center}
 \caption{}\label{fig12}
 \end{figure}

 By CSPR, we have
 $$s_{322/11} \cdot s_{11} = s_{432/11} + s_{4221/11} + s_{333/11} + s_{3321/11} + s_{32211/11} - s_{422/1} - s_{332/1} - s_{3221/1},$$
 as shown by Figure \ref{fig15}.

 \begin{figure}[ht!]
 \begin{center}
  \includegraphics[height = 1.6cm]{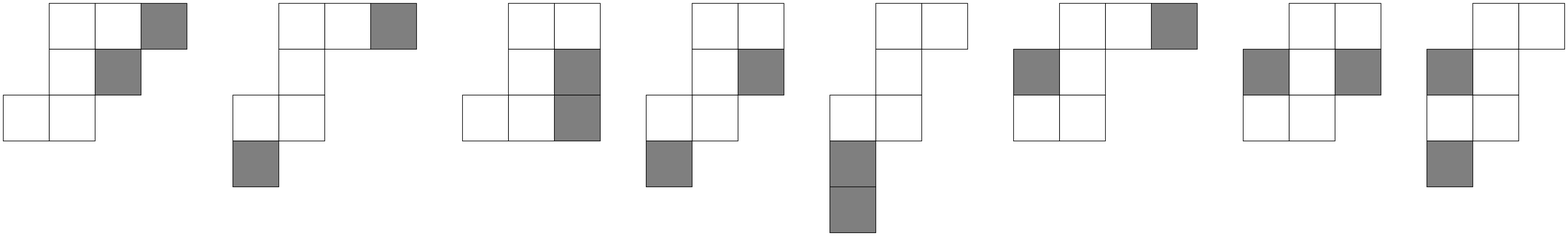}
 \end{center}
 \caption{}\label{fig15}
 \end{figure}

 By SMNR, we have
 $$s_{433/22} \cdot p_3 = s_{733/22} - s_{553/22} + s_{4333/22} - s_{43321/22} + s_{433111/22} + s_{433/1},$$
 as shown by Figure \ref{fig13}.

 \begin{figure}[ht!]
 \begin{center}
  \includegraphics[height = 2.4cm]{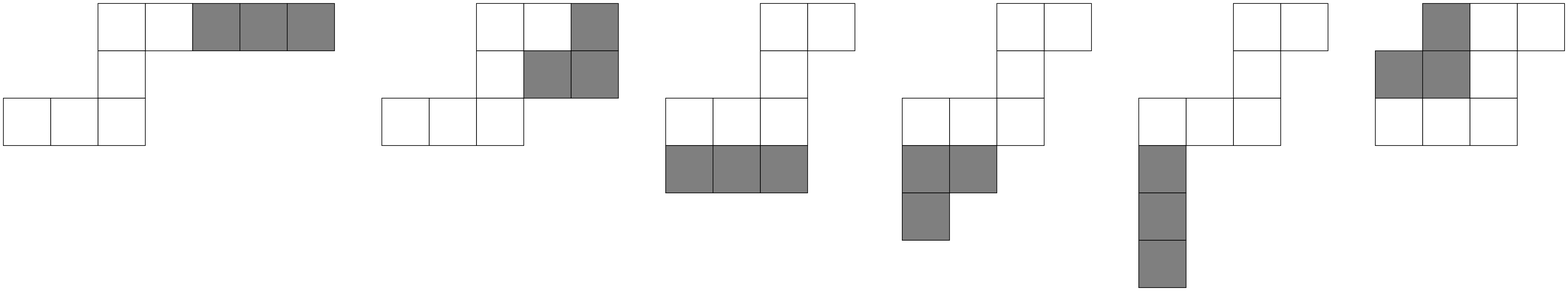}
 \end{center}
 \caption{}\label{fig13}
 \end{figure}

\end{exm}

Note that while the Pieri rule and the Murnaghan-Nakayama rule give the expansion in terms of a basis, their skew versions give only one possible (but obviously special) expansion in terms of skew Schur functions, which are not a basis of the space of symmetric functions.

\medskip

Assaf and McNamara provide an elegant bijective proof of their skew Pieri rule (but not of the skew Murnaghan-Nakayama rule; see Section \ref{final}). We describe this rule in detail in Section \ref{amcn} since an extension of it proves our main result.

\medskip

Define \emph{quantum power sum symmetric functions} by
$$\wp_r = \sum_{\tau \vdash r} (-1)^{\ell(\tau)-1}(q-1)^{\ell(\tau)-1} m_\tau,$$
$$\wp_\mu = \wp_{\mu_1} \wp_{\mu_2} \cdots.$$
For example,
$$\wp_4 = m_4 - (q-1) m_{31} - (q-1) m_{22} + (q-1)^2 m_{211} - (q-1)^3 m_{1111}$$
and
$$\wp_{22} = m_4 - 2(q-1)m_{31} + (q^2-2q+3) m_{22} + 2(q-1)(q-2) m_{211} + 6 (q-1)^2 m_{1111}.$$

The functions $\wp_\mu$ have connections with representation theory (more precisely, characters of the Hecke algebra of type A; see for example \cite[Theorem 6.5.3]{thesis}).

\medskip

We have
$$\wp_r|_{q = 1} = m_r = p_r, \qquad  \wp_r|_{q = 0} = \sum_{\tau \vdash r} m_\tau = s_r, \qquad \lim_{q \to \infty} \frac{\wp_r}{q^{r-1}} = (-1)^{r-1} m_{1^r} = (-1)^{r-1} s_{1^r}.$$

There exists a natural generalization of the Murnaghan-Nakayama rule, the \emph{quantum Murnaghan-Naka\-yama rule} (QMNR):
$$s_{\lambda} \cdot \wp_r = (-1)^{r+1} \sum_{\lambda^+} (-1)^{\wt(\lambda^+/\lambda)} q^{\hgt(\lambda^+/\lambda)} (q-1)^{\rib(\lambda^+/\lambda)-1} s_{\lambda^+},$$
where the internal sum on the right is over $\lambda^+$ such that $\lambda^+/\lambda$ is a broken ribbon of size $r$. See for example \cite[Theorem 6.5.2]{thesis} for a slightly different version.

\medskip

The following is our main result, the skew quantum Murnaghan-Nakayama rule.

\begin{thm}[SQMNR]
 For partitions $\lambda,\mu$, $\mu \subseteq \lambda$, and $r \geq 0$, we have
 $$s_{\lambda/\mu} \cdot \wp_r = \sum_{j=0}^r (-1)^{r+1-j} \! \sum_{\lambda^+,\mu^-} (-1)^{\wt(\lambda^+/\lambda)+\hgt(\mu/\mu^-)} q^{\hgt(\lambda^+/\lambda)+\wt(\mu/\mu^-)} (q-1)^{\rib(\lambda^+/\lambda)+\rib(\mu/\mu^-)-1} s_{\lambda^+/\mu^-},$$
 where the internal sum on the right is over $\lambda^+,\mu^-$ such that $\lambda^+/\lambda$ is a broken ribbon of size $r-j$, and $\mu/\mu^-$ is a broken ribbon of size $j$.
\end{thm}

There is another version of the statement that will be slightly more useful for our purposes.

\begin{thm}[SQMNR']
 For partitions $\lambda,\mu$, $\mu \subseteq \lambda$, and $r \geq 0$, we have
 $$s_{\lambda/\mu} \cdot \wp_r = \sum_{\lambda^+,\mu^-} (-1)^{|\mu/\mu^-|}(-q)^{\hgt(\lambda^+/\lambda)+\wt(\mu/\mu^-)} (1-q)^{\rib(\lambda^+/\lambda)+\rib(\mu/\mu^-)-1} s_{\lambda^+/\mu^-},$$
 where the sum on the right is over $\lambda^+,\mu^-$ such that $\lambda^+/\lambda$ and $\mu/\mu^-$ are broken ribbons with $|\lambda^+/\lambda| + |\mu/\mu^-| = r$.
\end{thm}

To see that these two versions are equivalent, note that
$$q^{\hgt(\lambda^+/\lambda)+\wt(\mu/\mu^-)} (q-1)^{\rib(\lambda^+/\lambda)+\rib(\mu/\mu^-)-1}=$$
$$= (-1)^{\hgt(\lambda^+/\lambda)+\wt(\mu/\mu^-)+\rib(\lambda^+/\lambda)+\rib(\mu/\mu^-)-1} (-q)^{\hgt(\lambda^+/\lambda)+\wt(\mu/\mu^-)} (1-q)^{\rib(\lambda^+/\lambda)+\rib(\mu/\mu^-)-1},$$
which means that the sign of $(-q)^{\hgt(\lambda^+/\lambda)+\wt(\mu/\mu^-)} (1-q)^{\rib(\lambda^+/\lambda)+\rib(\mu/\mu^-)-1}$ of a term on the right-hand side of SQMNR is
$$(-1)^{r+1-j+\wt(\lambda^+/\lambda)+\hgt(\mu/\mu^-)+\hgt(\lambda^+/\lambda)+\wt(\mu/\mu^-)+\rib(\lambda^+/\lambda)+\rib(\mu/\mu^-)-1}.$$
If $\pi/\sigma$ is a ribbon, we have $\wt(\pi/\sigma)+\hgt(\pi/\sigma)+1 = |\pi/\sigma|$. Therefore if $\pi/\sigma$ is a broken ribbon,
\begin{equation} \label{sum}
 \wt(\pi/\sigma)+\hgt(\pi/\sigma)+\rib(\pi/\sigma) = |\pi/\sigma|.
\end{equation}
That means that the sign above is equal to
$$(-1)^{r+1-j+|\lambda^+/\lambda|+|\mu/\mu^-|-1}=(-1)^{2r+j} = (-1)^{|\mu/\mu^-|}.$$

\medskip

The main theorem is a generalization of several statements. The following is a sample:
\begin{itemize}
 \item $q = 0$: a term on the right-hand side of SQMNR' is non-zero if and only if $\hgt(\lambda^+/\lambda)+\wt(\mu/\mu^-) = 0$. In this case, $\lambda^+/\lambda$ has height $0$ (and is a horizontal strip) and $\mu/\mu^-$ has width $0$ (and is a vertical strip). As noted above, $\wp_r|_{q = 0} = s_r$. SQMNR' specializes to the skew Pieri rule due to Assaf-McNamara \cite{amcn}.
 \item $q = 1$: a term on the right-hand side of SQMNR' is non-zero if and only if $\rib(\lambda^+/\lambda)+\rib(\mu/\mu^-)-1 = 0$. In this case, one of $\lambda^+/\lambda$ and $\mu/\mu^-$ is empty, and the other one is a ribbon. As noted above, $\wp_r|_{q = 1} = p_r$. SQMNR' therefore states
 $$s_{\lambda/\mu} \cdot p_r = \sum_{\lambda^+} (-1)^{\hgt(\lambda^+/\lambda)}s_{\lambda^+/\mu} + \sum_{\mu^-} (-1)^{k}(-1)^{\wt(\mu/\mu^-)} s_{\lambda/\mu^-} = $$
 $$= \sum_{\lambda^+} (-1)^{\hgt(\lambda^+/\lambda)}s_{\lambda^+/\mu} - \sum_{\mu^-} (-1)^{\hgt(\mu/\mu^-)} s_{\lambda/\mu^-},$$
 where the first sum is over $\lambda^+$ so that $\lambda^+/\lambda$ is a ribbon, and the second sum is over $\mu^-$ so that $\mu/\mu^-$ is a ribbon. This is the skew Murnaghan-Nakayama rule due to Assaf-McNamara \cite{slides}.
 \item $q \to \infty$: divide SQMNR by $q^{r-1}$ and send $q \to \infty$. The limit of the left-hand side is $(-1)^{r-1} s_{\lambda/\mu} s_{1^r}$. A term on the right is
 $$\lim_{q \to \infty} (-1)^{r+1-j} (-1)^{\wt(\lambda^+/\lambda)+\hgt(\mu/\mu^-)} \frac{q^{\hgt(\lambda^+/\lambda)+\wt(\mu/\mu^-)} (q-1)^{\rib(\lambda^+/\lambda)+\rib(\mu/\mu^-)-1}}{q^{r-1}}=$$
 $$= (-1)^{r+1-j} (-1)^{\wt(\lambda^+/\lambda)+\hgt(\mu/\mu^-)} \lim_{q \to \infty} q^{-(\wt(\lambda^+/\lambda)+\hgt(\mu/\mu^-))},$$
 where we used \eqref{sum}. This is non-zero if and only if $\wt(\lambda^+/\lambda)+\hgt(\mu/\mu^-) = 0$, i.e.\ if $\lambda^+/\lambda$ is a vertical strip and $\mu/\mu^-$ is a horizontal strip, and the limit is $(-1)^{r-1} (-1)^j$. SQMNR therefore implies the conjugate skew Pieri rule.
 \item $\mu = \emptyset$: SQMNR is obviously the quantum Murnaghan-Nakayama rule.
 \item $\mu = \emptyset$, $q = 0$: this is the classical Pieri rule.
 \item $\mu = \emptyset$, $q = 1$: this is the classical Murnaghan-Nakayama rule.
 \item $\mu = \emptyset$, $q \to \infty$: this implies the classical conjugate Pieri rule.
 \item $\lambda = \mu = \emptyset$: this gives the expansion of quantum power sum functions in the basis of Schur functions. The only Young diagrams of size $r$ that are also broken ribbons are hooks, i.e.\ diagrams of partitions of the type $(k,1^{r-k})$ for $1 \leq k \leq r$. Therefore (as we will verify independently in Lemma \ref{lemma1}),
 $$\wp_r = \sum_{k=1}^{r}(-q)^{r-k} s_{k,1^{r-k}}.$$
\end{itemize}

Define a \emph{broken ribbon tableau of shape $\lambda/\mu$ and type $\tau$} (respectively, \emph{reverse broken ribbon tableau of shape $\lambda/\mu$ and type $\tau$}) as an assignment of positive integers to the squares of $\lambda/\mu$ satisfying the following;
\begin{itemize}
 \item every row and column is weakly increasing (respectively, weakly decreasing);
 \item the integer $i$ appear $\tau_i$ times;
 \item the set $T_i$ of squares occupied by $i$ forms a broken ribbon or is empty.
\end{itemize}
For a (reverse) broken ribbon tableau $T$ we define $\hgt(T) = \sum \hgt(T_i)$, $\wt(T) = \sum \wt(T_i)$, $\rib(T) = \sum \rib(T_i)$.

\medskip

The main theorem implies the following corollary.

\begin{cor}
 We have
 $$s_{\lambda/\mu} \cdot \wp_\tau = \sum_{\stackrel{\lambda^+ \supseteq \lambda}{\scriptscriptstyle \mu^- \subseteq \mu}} (-1)^{|\mu/\mu^-|} \chi(\lambda^+,\lambda,\mu,\mu^-;\tau) s_{\lambda^+/\mu^-},$$
 where
 $$\chi(\lambda^+,\lambda,\mu,\mu^-;\tau) = \sum (-q)^{\hgt(T')+\wt(T'')} (1-q)^{\rib(T')+\rib(T'')-1}$$
 with the sum over all pairs $(T',T'')$ of a broken ribbon tableau and a reverse broken ribbon tableau of shapes $\lambda^+/\lambda$ and $\mu/\mu^-$, respectively, and types $\tau'$ and $\tau''$, respectively, so that $\tau'+\tau'' = \tau$.
\end{cor}

The paper is structured as follows. In Section \ref{amcn}, we describe the sign-reversing involution of Assaf and McNamara that was used to prove their skew Pieri rule. Furthermore, we show a variant of this involution that proves the conjugate skew Pieri rule. Note that this involution is actually much simpler than the one in \cite{amcn} (but, of course, does not provide a bijective proof of the skew Pieri rule itself). In Section \ref{inv}, we present an extension of these involutions that proves the skew quantum Murnaghan-Nakayama rule. There is quite some work involved to interpret the right-hand side of SQMNR in an appropriate way, but once this is done the involution is just a natural combination of the two involutions in Section \ref{amcn}. In Section \ref{another}, we present another proof of SQMNR, via the skew Littlewood-Richardson rule of Lam-Lauve-Sotille \cite{lls}; since this result (at the moment) only has an algebraic proof, this proof of SQMNR is not completely combinatorial. In Section \ref{hlp}, we give some conjecured skew Pieri-type rules for Hall-Littlewood polynomials, for which our combinatorial methods seem to fail. We finish with some concluding remarks in Section \ref{final}.

\section{Proofs of the skew Pieri rule and its dual} \label{amcn}

One of the most important algorithms on semistandard Young tableaux is the Robinson-Schensted \emph{row insertion}. Given a semistandard Young tableau $T$ of shape $\lambda$ and an integer $k$, we can \emph{insert $k$ into $T$} as follows. Define $k_1 = k$. Find the smallest $j$ so that $T_{1j} > k_1$, replace $T_{1j}$ by $k_1$, and define $k_2$ to be the previous value of $T_{1j}$. Then find the smallest $j$ so that $T_{2j} > k_2$, replace $T_{2j}$ by $k_2$, and define $k_3$ to be the previous value of $T_{2j}$. Continue until, for some $i'$, all elements of row $i'$ are $\leq k_{i'}$. Then define $T_{i',\lambda_i'+1} = k_{i'}$, and finish the algorithm. The result is again a semistandard Young tableau. We say that the insertion of $k$ into $T$ \emph{exits in row $i'$}. See \cite[\S 7.11]{stanley} for details.

\begin{exm}
 Inserting $1$ into the tableau on the left of Figure \ref{fig2} produces the tableau on the right.
 \begin{figure}[ht!]
 \begin{center}
  \includegraphics[height=2.5cm]{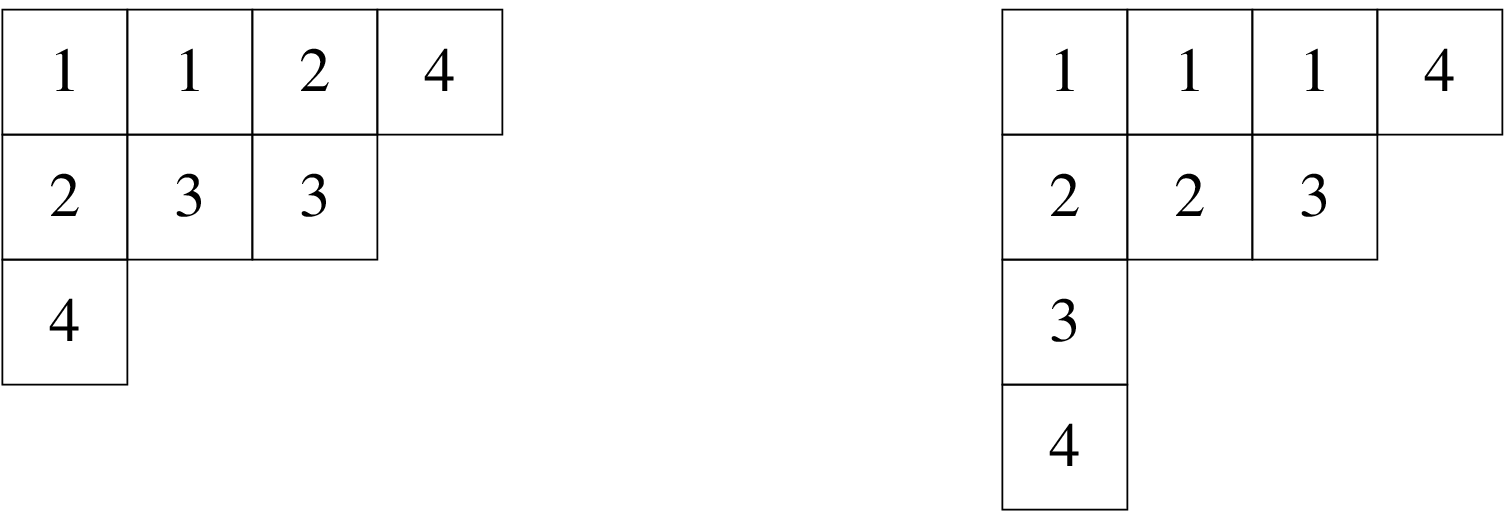}
 \end{center}
 \caption{}\label{fig2}\end{figure}
\end{exm}

\medskip

Now assume we have a \emph{skew} semistandard Young tableau $T$ of some shape $\lambda/\mu$. We can \emph{insert $k$ into $T$} for some integer $k$ in almost exactly the same way. Define $k_1 = k$. Find the smallest $j$, $\mu_1 < j \leq \lambda_1$, so that $T_{1j} > k_1$, replace $T_{1j}$ by $k_1$, and define $k_2$ to be the previous value of $T_{1j}$. Then find the smallest $j$, $\mu_2 < j \leq \lambda_2$, so that $T_{2j} > k_2$, replace $T_{2j}$ by $k_2$, and define $k_3$ to be the previous value of $T_{2j}$. Continue until, for some $i'$, all elements of row $i'$ are $\leq k_{i'}$. Then define $T_{i',\lambda_{i'}+1} = k_{i'}$, and finish the algorithm. The result is again a semistandard Young tableau. We say that the insertion of $k$ into $T$ \emph{exits in row $i'$}.

\medskip

There is, however, another natural kind of insertion. Take $i_0$ so that either $i_0 = 1$ or $\mu_{i_0-1} > \mu_{i_0}$, and take $k_{i_0+1} = T_{i_0,\mu_{i_0}+1}$. We can \emph{insert from row $i_0$ in $T$} as follows. Erase the entry $T_{i_0,\mu_{i_0}+1}$. Find the smallest $j$, $\mu_{i_0+1} < j \leq \lambda_{i_0+1}$, so that $T_{i_0+1,j} > k_{i_0+1}$, replace $T_{i_0+1,j}$ by $k_{i_0+1}$, and define $k_{i_0+2}$ to be the previous value of $T_{i_0+1,j}$. Then find the smallest $j$, $\mu_{i_0+2} < j \leq \lambda_{i_0+2}$, so that $T_{i_0+2,j} > k_{i_0+2}$, replace $T_{i_0+2,j}$ by $k_{i_0+2}$, and define $k_{i_0+3}$ to be the previous value of $T_{i_0+2,j}$. Continue until, for some $i'$, all elements of row $i'$ are $\leq k_{i'}$. Then define $T_{i',\lambda_{i'}+1} = k_{i'}$, and finish the algorithm. The result is again a semistandard Young tableau. We say that the insertion from row $i_0$ in $T$ \emph{exits in row $i'$}.

\begin{exm}
 In the following figures, we have an insertion of $1$ into a tableau, and insertion from row $2$ in a tableau.
 \begin{figure}[ht!]
 \begin{center}
  \includegraphics[height=2.5cm]{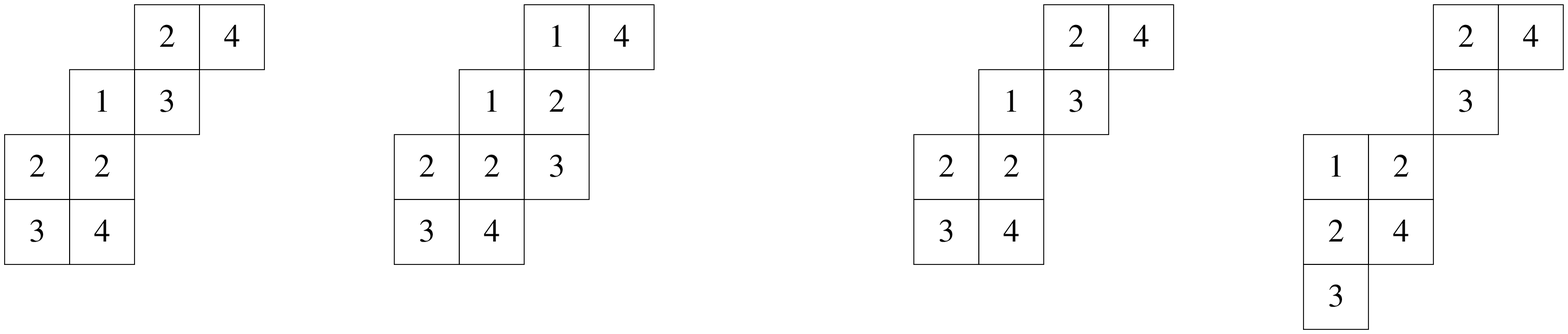}
 \end{center}
 \caption{}\label{fig3}\end{figure}
\end{exm}

Note that insertion into $T$ is in a way a special case of insertion from a row in $T$. Indeed, take $\mu_0 = \lambda_1$, $\lambda_0 = \lambda_1 + 1$, and define $T_{0,\lambda_1} = k$. Then insertion from row $0$ in the new tableau gives the same result as insertion of $k$ into the original tableau.

\medskip

Insertion has an inverse operation, \emph{reverse insertion}. Say we are given a semistandard Young tableau $T$ of shape $\lambda/\mu$. Take $i'$ so that $\lambda_{i'+1} < \lambda_{i'}$. We \emph{reverse insert from row $i'$ in $T$} as follows. Define $k_{i'-1} = T_{i',\lambda_{i'}}$. Erase the entry $T_{i',\lambda_{i'}}$. Find the largest $j$, $\mu_{i'-1} < j \leq \lambda_{i'-1}$, so that $T_{i'-1,j} < k_{i'-1}$, replace $T_{i'-1,j}$ by $k_{i'-1}$, and define $k_{i'-2}$ to be the previous value of $T_{i'-1,j}$. Then find the largest $j$, $\mu_{i'-2} < j \leq \lambda_{i'-2}$, so that $T_{i'-2,j} < k_{i'-2}$, replace $T_{i'-2,j}$ by $k_{i'-2}$, and define $k_{i'-3}$ to be the previous value of $T_{i'-2,j}$. Continue until we have $k_{i_0}$, where either $i_0 = 0$ or all elements of row $i_0$ are $\geq k_{i_0}$. If $i_0 = 0$, the result is a pair $(S,k)$, where $S$ is a semistandard Young tableau and $k = k_0$. We call $k$ the \emph{exiting integer}. If $i_0 \geq 1$ and all elements of row $i_0$ are $\geq k_{i_0}$, define $T_{\mu_{i_0}} = k_{i_0}$. The result is a semistandard Young tableau $S$. We say that the reverse insertion from row $i'$ in $T$ \emph{exits in row $i_0$}.

\begin{exm}
 In the following figures, we have reverse insertion from rows $2$ (which exits in row $0$ with exiting integer $2$) and $4$ (which exits in row $1$).
 \begin{figure}[ht!]
 \begin{center}
  \includegraphics[height=2cm]{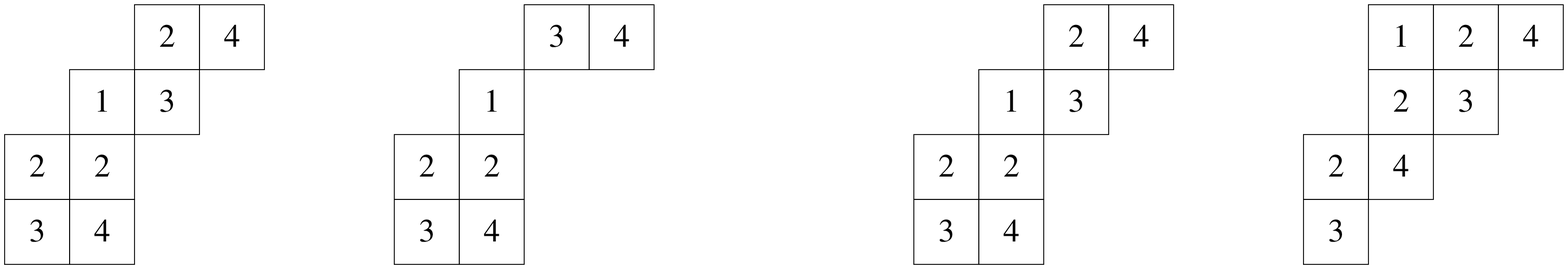}
 \end{center}
 \caption{}\label{fig4}\end{figure}

\end{exm}

In \cite{amcn}, the operations of insertion and reverse insertion are proved to be inverses of one another in the following sense. If the insertion of an integer $k$ into a semistandard Young tableau $T$ exits in row $i'$ and the resulting tableau is $S$, then the reverse insertion from row $i'$ in $S$ exits in row $0$ and the result is $(T,k)$. If the insertion from row $i_0$ into $T$ exits in row $i'$ and the resulting tableau is $S$, then the reverse insertion from row $i'$ in $S$ exits in row $i_0$ and the result is $T$. Similarly, if the reverse insertion from row $i'$ in $T$ exits in row $0$ and the result is $(S,k)$, then the insertion of $k$ into $S$ exits in row $i'$ and the result is $T$. And if the reverse insertion from row $i'$ in $T$ exits in row $i_0 \geq 1$ and the result is $S$, then the insertion from row $i_0$ into $S$ exits in row $i'$ and the result is $T$.

\medskip

We will also need the following property of insertion and reverse insertion. The lemma essentially states that insertions never cross.

\begin{lemma} \label{insertion}
 Say we are given a semistandard Young tableau $T$.
 \begin{enumerate}
  \renewcommand{\labelenumi}{(\alph{enumi})}
  \item If $S$ is obtained by reverse insertion from $i'$ in $T$ that exits in row $i_0 > 0$, and $R$ is obtained by reverse insertion from $i''<i'$ in $S$ that exits in row $i_0'$, then $i_0' < i_0$.
  \item If $S$ is obtained by reverse insertion from $i'$ in $T$ that exits in row $0$ with exiting integer $k'$, and $R$ is obtained by reverse insertion from $i''<i'$ in $S$ that exits in row $i_0'$, then $i_0 = 0$ and the reverse insertion exits with exiting integer $k'' > k'$.
  \item If reverse insertion from $i'$ in $T$ exits in $i_0$ and insertion from $i_0' > i_0$ in $T$ exits in $i''$,then $i'' > i'$.
 \end{enumerate}
\end{lemma}
\begin{proof}
 (a) If $i_0 \geq i''$, then $i_0' < i'' \leq i_0$ and the claim follows. Assume $i_0 < i''$. We claim that if the reverse insertion from $i'$ in $T$ passes through $(i,j')$ and the reverse insertion from $i''$ in $S$ passes through $(i,j'')$, then $j'' \geq j'$; in other words, reverse insertion from $T$ lies weakly to the right of the reverse insertion from $S$. The statement is true for $i = i''$ because in this case, $j'' = \lambda_i$. If it holds for $i$ and $j' < j''$, the reverse insertion from $i'$ in $T$ bumps the entry $T(i,j')$ into row $i-1$; then the reverse insertion from $i''$ in $S$ bumps the entry $T(i,j'') \geq T(i,j')$ into a position which cannot be the the left of the new position of $T(i,j')$ in row $i-1$. If, on the other hand, $j' = j''$, the reverse insertion from $i'$ in $T$ again bumps the entry $T(i,j')$ into row $i-1$ and is itself replaced by a strictly larger entry. Then reverse insertion from $i''$ in $S$ bumps this strictly larger entry into the next row into a position which cannot be the the left of the new position of $T(i,j')$ in row $i-1$.\\
 This means that the reverse insertion from row $i''$ in $S$ passes through row $i_0$ and so it exits in row $< i_0$.
 (b) By the reasoning in (a), the reverse insertion from $S$ is weakly to the right of the reverse insertion from $T$. In particular, reverse insertion from $S$ reaches row $1$, and if the exiting integer $k'$ is bumped from position $(1,j')$, then the exiting integer $k''$ is bumped from $(1,j'')$ for $j'' \geq j'$. In particular, $k'' > k'$.\\
 (c) We claim that if reverse insertion from $i'$ in $T$ passes through $(i,j') \in [T]$, where $i_0' \leq i \leq i'$, then insertion from $i_0'$ in $T$ passes through the cell $(i,j'') \in [T]$ for some $j'' \leq j'$. The statement is true for $i = i_0'$ because in that case, $j'' = \mu_i+1 \leq j'$. If it holds for $i$, then the entry from row $i+1$, say $a$, that was bumped into row $i$ during the reverse insertion from $i'$ in $T$, must be $<a$, and lies in position $(i,j')$ in $T$. Therefore $T(i,j'') < a$ and cannot be bumped into a position to the right of $a$ in row $i+1$ in $T$.\\
 In particular, insertion from row $i_0'$ in $T$ passes through row $i'$, and so the insertion exits in row $i'' > i'$.
\end{proof}

The involution by Assaf and McNamara which proves the skew Pieri rule works as follows. Say we are given a skew shape $\lambda/\mu$ and a semistandard Young tableau $T$ of shape $\lambda^+/\mu^-$, where $\lambda^+/\lambda$ is a horizontal strip and $\mu/\mu^-$ is a vertical strip. Let $v$ be the empty word. Let $i=\infty$ if $\mu = \mu^-$, and let $i$ be the top row of $\mu/\mu^-$ otherwise.

\medskip

While $\lambda^+ \neq \lambda$ and the reverse insertion from row $i'$, the top row of $\lambda^+/\lambda$, in $T$ exits in row $0$ and results in $(S,k)$, attach $k$ to the beginning of $v$, let $T = S$, and let $\lambda^+/\mu^-$ be the shape of the new $T$ (note that $\lambda^+_{i'}$ is decreased by $1$ and $\mu^-$ remains the same).

\medskip

If the while loop stops when $\lambda^+ \neq \lambda$ and the reverse insertion from row $i'$ in $T$ exits in row $i_0$, $0 < i_0 < i$, and results in $S$, let $T = S$.

\medskip

If the while loop stops when $\lambda^+ = \lambda$, $\mu \neq \mu^-$, or when $\lambda^+ \neq \lambda$ and the reverse insertion from row $i'$ in $T$ exits in row $i_0$, $i_0 \geq i$, insert from row $i$ into $T$ and call the resulting tableau $T$.

\medskip

Finish the algorithm by inserting the entries of $v$ from left to right into $T$. The final result is a semistandard Young tableau of some shape $\lambda^{++}/\mu^{--}$, we denote it $\Phi_{\lambda,\mu,\lambda^+,\mu^-}(T)$.

\begin{exm}
 The left drawing shows a skew semistandard Young tableau with $\lambda^+ = 8855432$, $\lambda = 855533$, $\mu = 43222$, $\mu^- = 42111$. The while loop changes $v$ to $2445$ and it stops because after four reverse insertions, the next reverse insertion (from row $7$) exits in row $1$ (see the second drawing). Since this is strictly above the top row of $\mu/\mu^-$, i.e.\ $2$, we also perform this reverse insertion from row $7$ (see the third drawing). Then we insert the integers $2$, $4$, $4$ and $5$ and we get the skew semistandard Young tableau pictured on the right, with $\lambda^+ = 8855431$, $\lambda = 855533$, $\mu = 43222$, $\mu^- = 32111$.\\

 \begin{figure}[ht!]
 \begin{center}
  \includegraphics[height=3.5cm]{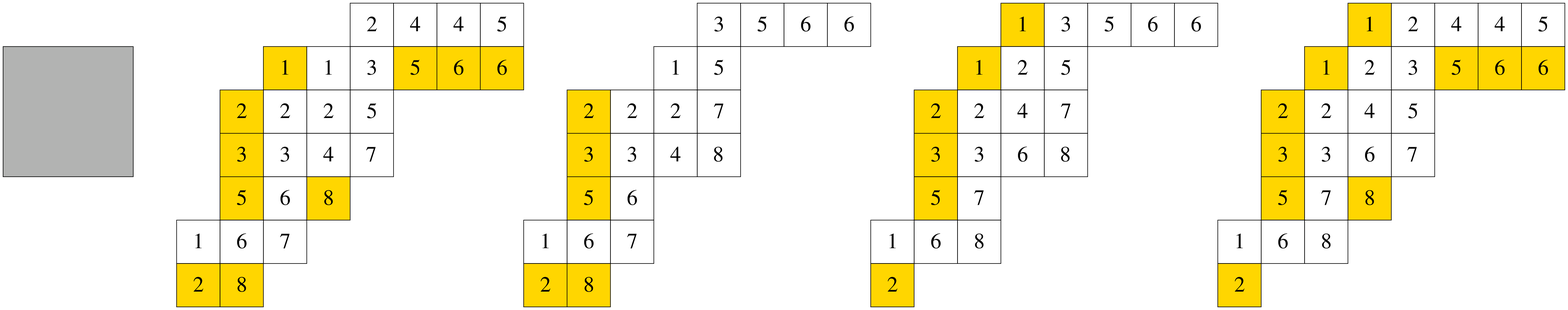}
 \end{center}
 \caption{}\label{fig5}\end{figure}

 In the second example, we start with $\lambda^+ = 8855431$, $\lambda = 855533$, $\mu = 43222$, $\mu^- = 42111$, see the left drawing. The while loop again changes $v$ to $2445$ and it stops because after four reverse insertions, the next reverse insertion (from row $7$) exits in row $5$ (see the second drawing). Since this is not above the top row of $\mu/\mu^-$, we do not perform this reverse insertion. Instead, we insert from the top row of $\mu/\mu^-$, i.e.\ $2$ (see the third drawing). Then we insert the integers $2$, $4$, $4$ and $5$ and we get the skew semistandard Young tableau pictured on the right, with $\lambda^+ = 8855432$, $\lambda = 855533$, $\mu = 43222$, $\mu^- = 43111$.\\

 \begin{figure}[ht!]
 \begin{center}
  \includegraphics[height=3.5cm]{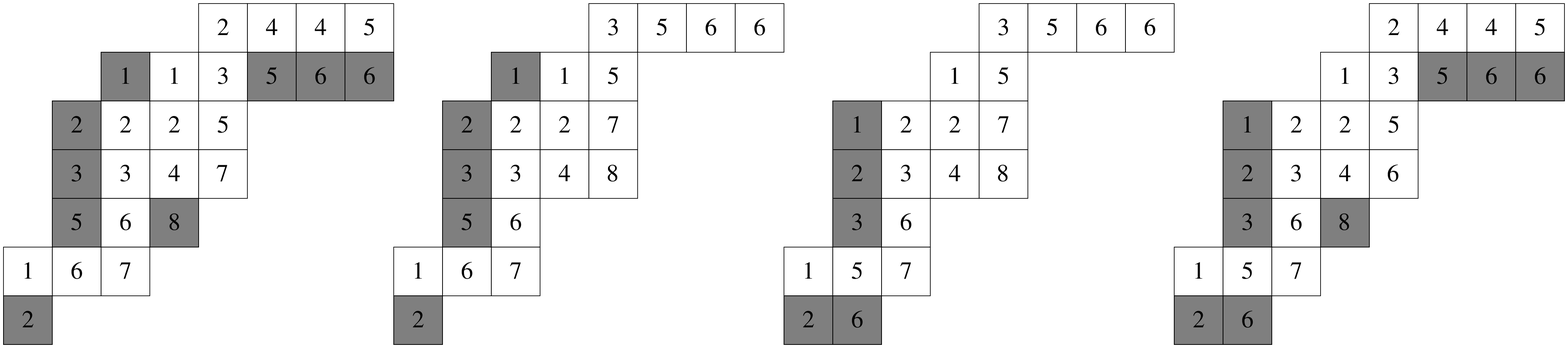}
 \end{center}
 \caption{}\label{fig6}\end{figure}

 In the third example, we start with $\lambda^+ = 9855331$, $\lambda = 855533$, $\mu = 43222$, $\mu^- = 43222$, see the left drawing. The while loop changes $v$ to $11245$ and it stops because after five reverse insertions, $\lambda^+ = \lambda$ and $\mu = \mu^-$. So we insert the integers $1$, $1$, $2$, $4$ and $5$ and we get the original skew semistandard Young tableau, pictured on the right.\\

 \begin{figure}[ht!]
 \begin{center}
  \includegraphics[height=3.5cm]{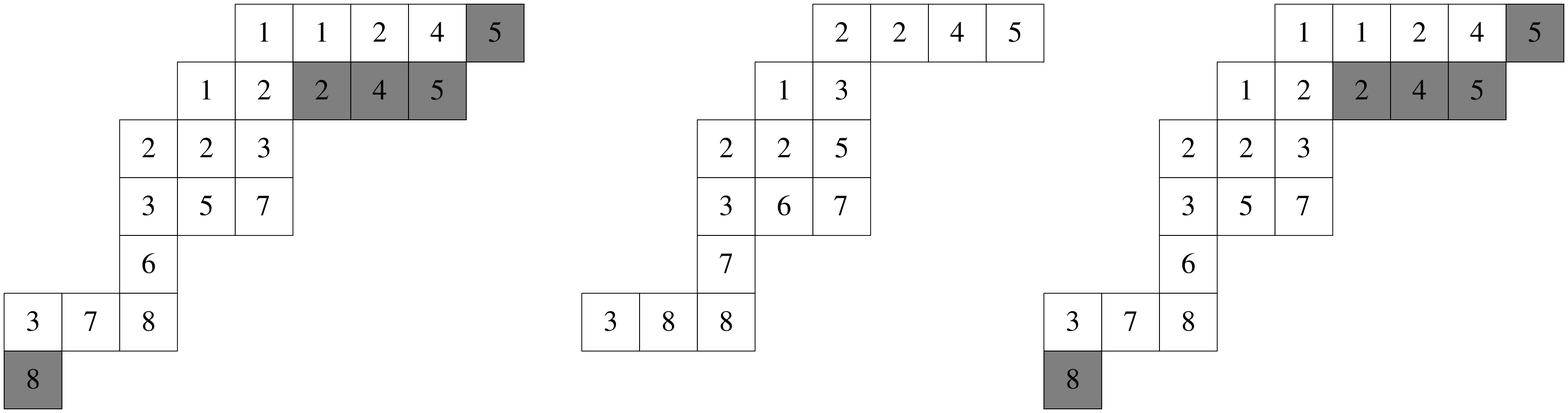}
 \end{center}
 \caption{}\label{fig7}\end{figure}
\end{exm}

It turns out that $\Phi$ is an involution, and $T$ is a fixed point if and only if $\mu = \mu^-$ and the while loop stops when $\lambda^+ = \lambda$. Such fixed points are in one-to-one correpondence with pairs $(S,v)$, where $S$ is a semistandard Young tableau of shape $\lambda/\mu$ and $v$ is a weakly increasing word. Indeed, if we stop the algorithm after the while loop, we have exactly such a pair, and given a pair $(S,v)$, we can insert the entries of $v$ from left to right into $S$ to get the corresponding $T$. Furthermore, if $T$ is not a fixed point, then $|\mu^{--}| = |\mu^-| \pm 1$. It is easy to see that this shows the skew Pieri rule. See \cite{amcn} for details and a precise proof.

\medskip

As mentioned in the introduction, the conjugate skew Pieri rule follows from SPR by applying the involution $\omega$ on the algebra of symmetric functions. There is, however, an involution in the spirit of Assaf-McNamara that proves CSPR.

\medskip

Fix $\lambda,\mu,r$. A term on the right-hand side is represented by a semistandard skew Young tableau of shape $\lambda^+/\mu^-$, where $\lambda^+/\lambda$ is a vertical strip, $\mu/\mu^-$ is a horizontal strip, and $|\lambda^+/\lambda|+|\mu/\mu^-| = r$. Such a tableau $T$ is weighted by $(-1)^{|\mu/\mu^-|}$. Let $i$ denote the bottom row of $\mu/\mu^-$ (unless $\mu = \mu^-$, in which case take $i = 0$). Now reverse insert from row $i'$, the bottom row of $\lambda^+/\lambda$, in $T$ (unless  $\lambda^+ = \lambda$). If the reverse insertion exits the diagram in row $\geq i$ (\emph{except} in the case when $\mu = \mu^-$ and the reverse insertion exits in row $0$), call this new diagram $\Psi(T)=\Psi_{\lambda,\mu,\lambda^+,\mu^-}(T)$. See Figure \ref{fig8}, left. If this reverse insertion exits the diagram in row $<i$, or if $\lambda^+ = \lambda$, insert from row $i$ in $T$ and call the result $\Psi(T)=\Psi_{\lambda,\mu,\lambda^+,\mu^-}(T)$. See Figure \ref{fig8}, middle. When $\mu = \mu^-$ and the reverse insertion exits in row $0$, take $\Psi(T) = \Psi_{\lambda,\mu,\lambda^+,\mu^-}(T) = T$. See Figure \ref{fig8}, right.

\begin{exm}
 For the skew semistandard Young tableau on the left of Figure \ref{fig8}, reverse insertion from row $9$ (the bottom row of $\lambda^+/\lambda$) exits in row $5$, which is weakly below the bottom row of $\mu/\mu^-$. Therefore we perform this reverse insertion, and the result is the left picture of Figure \ref{fig9}. For the skew semistandard Young tableau in the middle of Figure \ref{fig8}, reverse insertion from row $9$ (the bottom row of $\lambda^+/\lambda$) exits in row $4$, which is striclty above the bottom row of $\mu/\mu^-$. Therefore we insert from row $5$ (the bottom row of $\mu/\mu^-$), the result is the middle picture of Figure \ref{fig9}. For the skew semistandard Young tableau on the right of Figure \ref{fig8}, reverse insertion from row $9$ (the bottom row of $\lambda^+/\lambda$) exits in row $0$. This means that the tableau is a fixed point of $\Psi$. Therefore we perform this reverse insertion, and the result is the left picture of Figure \ref{fig9}. The right picture in Figure \ref{fig9} shows the skew semistandard Young tableau that we get if we repeatedly reverse insert from the bottom row of $\lambda^+/\lambda$; the exiting integers are $1,2,3,4,5,6$.

 \begin{figure}[ht!]
 \begin{center}
  \includegraphics[height=4.5cm]{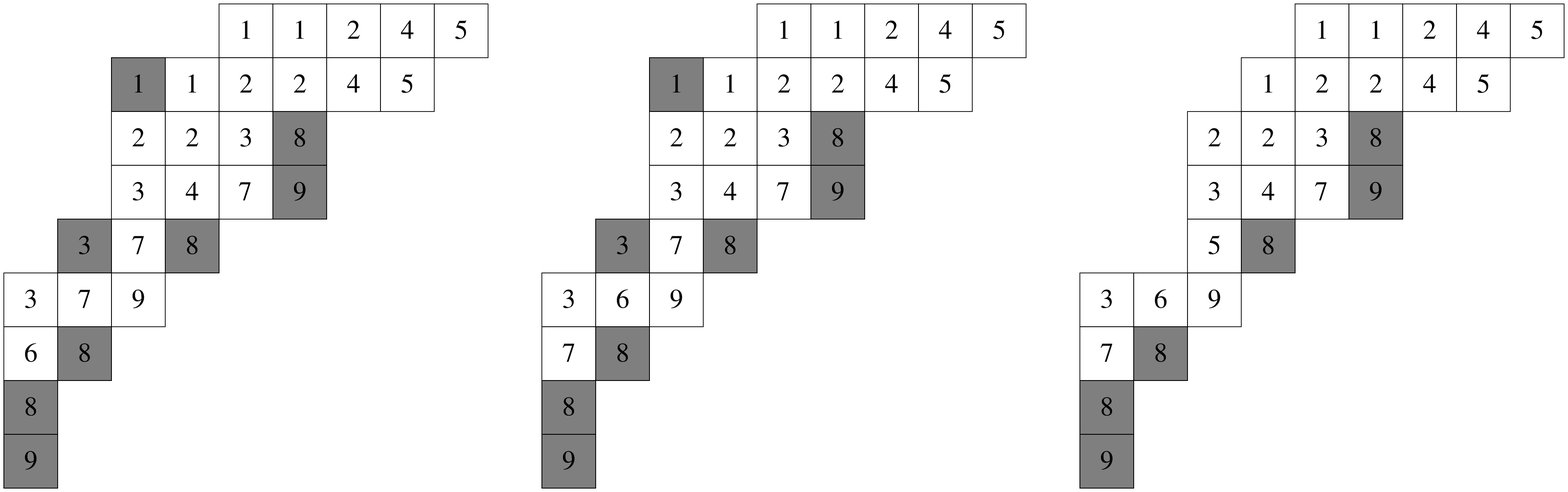}
 \end{center}
 \caption{}\label{fig8}\end{figure}

 \begin{figure}[ht!]
 \begin{center}
  \includegraphics[height=5cm]{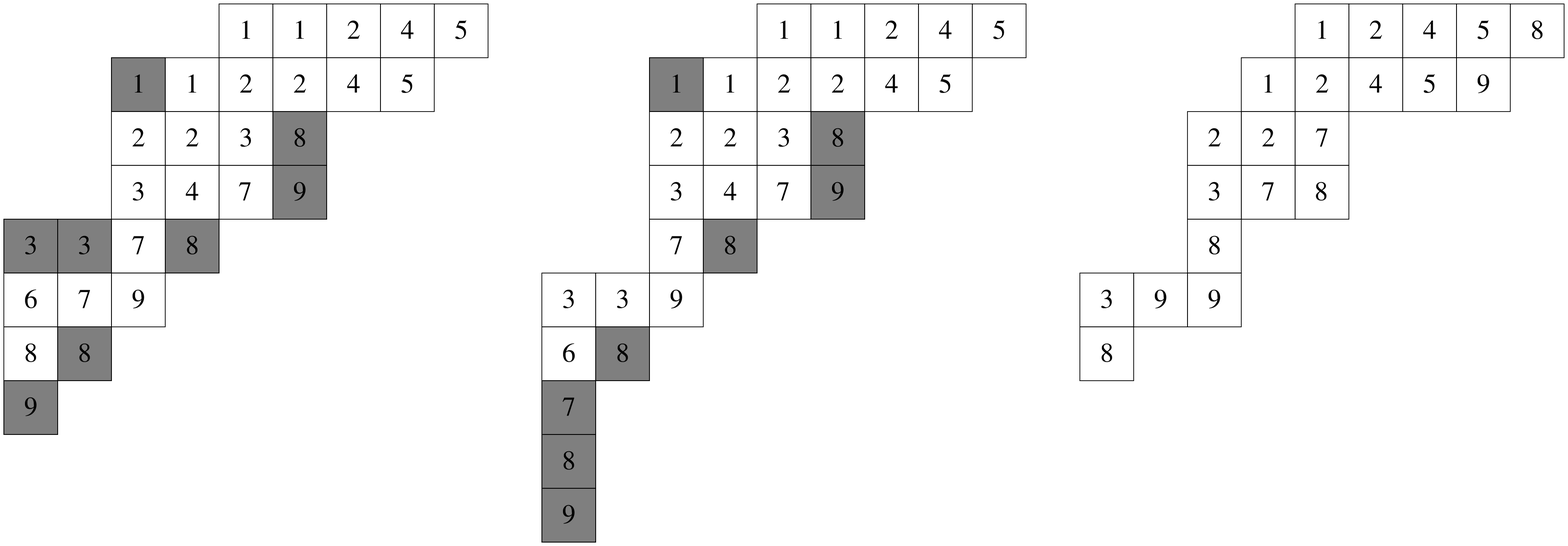}
 \end{center}
 \caption{}\label{fig9}\end{figure}
\end{exm}

\begin{prop}
 The map $\Psi_{\lambda,\mu,\lambda^+,\mu^-}$ is an involution that is sign-reversing except on fixed points. Furthermore, the fixed points are in a natural bijective correspondence with elements on the left-hand side of CSPR.
\end{prop}
\begin{proof}
 Say that $\lambda^+ \neq \lambda$ and the reverse insertion from $i'$, the bottom row of $\lambda^+/\lambda$, exits in row $i_0$, $0 \neq i_0 \geq i$, where $i$ is the bottom row of $\mu/\mu^-$, and results in $S$ of shape $\lambda^{++}/\mu^{--}$. Recall that in this case, $\Psi(T) = S$. The partition $\mu^{--}$ differs from $\mu^-$ only in row $i_0$, and $\mu^{--}_{i_0} = \mu^{-}_{i_0} - 1$. Also, $\lambda^{++}$ differs from $\lambda^{+}$ only in row $i'$, and $\lambda^{++}_{i'} = \lambda^{+}_{i'} - 1$. Note that the bottom row of $\mu/\mu^{--}$ is $i_0$. If $\lambda^{++} = \lambda$, then $\Psi(S)$ is obtained by inserting from row $i_0$ in $S$, which is $T$ (because insertion and reverse insertion are inverse operations). If $\lambda^{++} \neq \lambda$, then the bottom row of $\lambda^{++}/\lambda$ is strictly above $i'$; furthermore, reverse insertion from this row exits in row $< i_0$ by Lemma \ref{insertion}, part (a). So we also obtain $\Psi(S)$ by inserting from row $i_0$ in $S$, and we get $T$.\\
 Now assume that $\lambda^+ \neq \lambda$ and the reverse insertion from $i'$, the bottom row of $\lambda^+/\lambda$, exits in row $< i$. Then $S = \Psi(T)$ of shape $\lambda^{++}/\mu^{--}$ is the result of inserting from row $i$ in $T$, assume that this insertion exits in row $i''$. We know that $\mu^{--}$ differs from $\mu^-$ only in row $i$, $\mu^{--}_{i} = \mu^{-}_{i} + 1$, and $\lambda^{++}$ differs from $\lambda^{+}$ only in row $i''$, $\lambda^{++}_{i''} = \lambda^{+}_{i''} + 1$. By Lemma \ref{insertion}, part (c), $i'' > i'$. That means that when we perform $\Psi$ on $S$, we reverse insert from row $i''$ in $S$. The reverse insertion results in $T$ and exits in row $i$, which is weakly below the bottom row of $\mu/\mu^{--}$, so $\Psi(S) = T$.\\
 If $\lambda^+ = \lambda$, we obtain $S = \Psi(T)$ of shape $\lambda^{++}/\mu^{--}$ by inserting from row $i$ in $T$, say that the insertion exits in row $i'$. In $S$, $\lambda^+/\lambda$ has only one cell, which is in row $i'$. Furthermore, reverse insertion from row $i'$ in $S$ exits in row $i$, which is weakly below the bottom row of $\mu/\mu^{--}$. So the result of this reverse insertion, $T$, is also $\Psi(S)$.\\
 Finally, assume that $T$ is a fixed point, i.e.\ that $\mu = \mu^-$ and that the reverse insertion from row $i'$, the bottom row of $\lambda^+/\lambda$, exits in row $0$. Call the resulting tableau $T_1$ (of shape $\lambda^{++}/\mu$) and the exiting integer $k_1$. By Lemma \ref{insertion}, part (b), that means that if we again reverse insert from the bottom row of $\lambda^+_1/\lambda$ in $T_1$, the reverse insertion again exits in row $0$, and the exiting integer $k_2$ is strictly greater than $k_1$. Call the resulting tableau $T_2$, and continue. After $r$ steps, we have a semistandard Young tableau $S = T_r$ of shape $\lambda/\mu$, and a strictly decreasing word $w = k_r k_{r-1} \cdots k_1$. Such pairs $(S,w)$ are obviously enumarated by the left-hand side of CSPR.
\end{proof}

\section{A bijective proof of the main theorem} \label{inv}

The first step of our proof is to interpret the right-hand side of SQMNR' as a weighted sum over some combinatorial objects. The appropriate objects turn out to be skew semistandard Young tableaux with some cells colored gray. To motivate these colorings, observe the following. If we ``glue'' together a vertical strip and a horizontal strip in such a way that the result is a skew diagram, then this skew diagram cannot have any $2 \times 2$ squares. In other words, it is a broken ribbon. This also holds the other way around: if we are given a broken ribbon, we can break it up into a vertical strip and a horizontal strip. See Figure \ref{fig10} for two examples. Note that the right example is special: the white cells (i.e.\ the cells of $\mu$ and the cells of $\lambda/\mu$ we put in the horizontal strip) form a partition. In other words, the cells of the horizontal strip are never have cells of vertical strip to the left or above them.

\begin{figure}[ht!]
 \begin{center}
  \includegraphics[height=4cm]{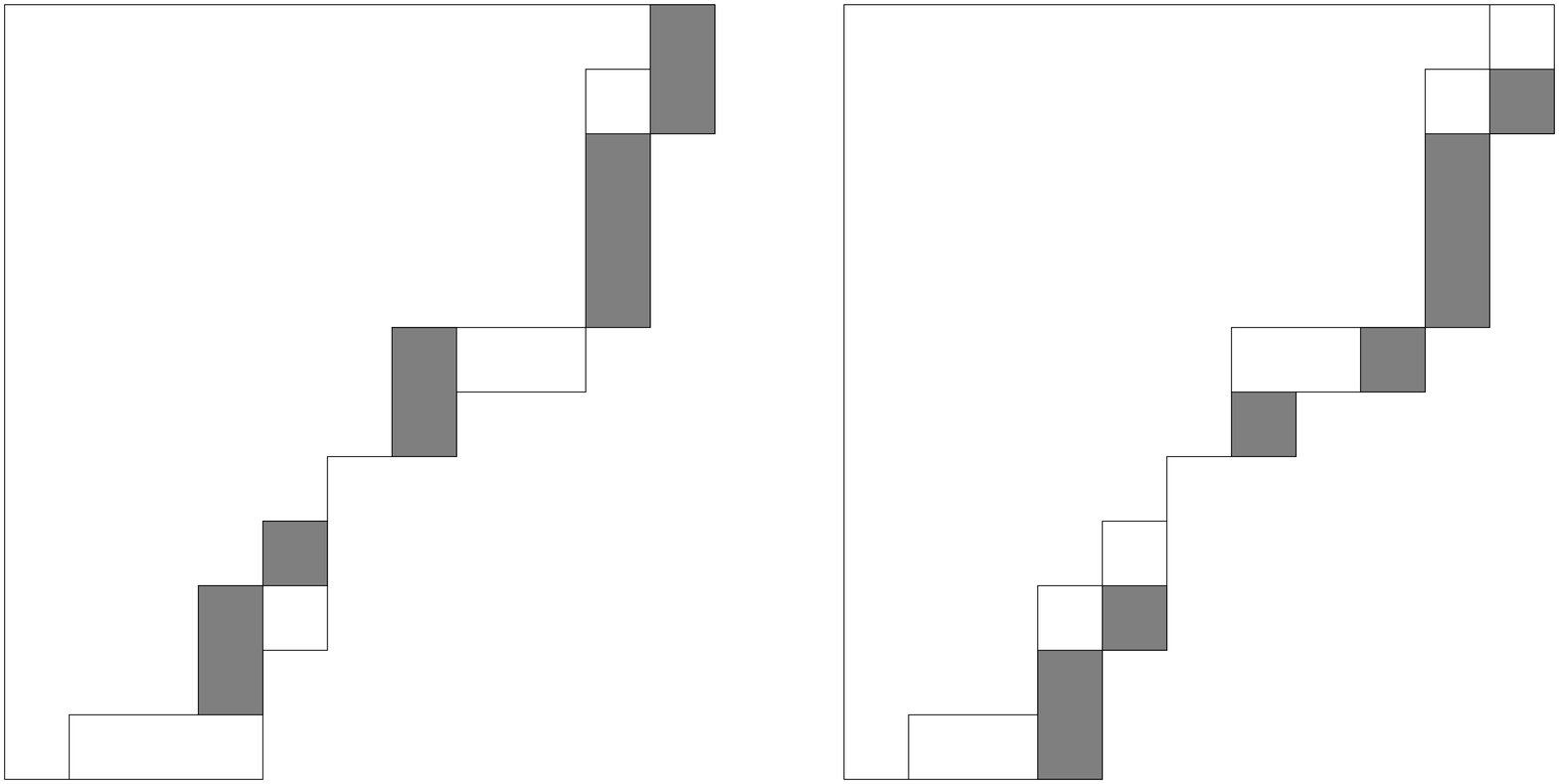}
 \end{center}
\caption{}\label{fig10}\end{figure}

\medskip

Let us multiply both sides of SQMNR' by $1-q$ and call this statement SQMNR'':

$$s_{\lambda/\mu} \cdot \left( \sum_{\tau \vdash r}
(1-q)^{\ell(\tau)} m_\tau\right) = \sum_{\lambda^+,\mu^-} (-1)^{|\mu/\mu^-|}(-q)^{\hgt(\lambda^+/\lambda)+\wt(\mu/\mu^-)} (1-q)^{\rib(\lambda^+/\lambda)+\rib(\mu/\mu^-)} s_{\lambda^+/\mu^-},$$

We have fixed $\lambda,\mu,r$. Say that we are given $\lambda^+,\mu^-$ such that $\lambda^+/\lambda$ and $\mu/\mu^-$ are broken ribbons with $|\lambda^+/\lambda| + |\mu/\mu^-| = r$, and a skew semistandrad Young tableau $T$ of shape $\lambda^+/\mu^-$. Our first goal is to break up each of the broken ribbons $\lambda^+/\lambda$ and $\mu/\mu^-$ into a vertical strip and a horizontal strip. More precisely, we wish to choose partitions $\lambda',\mu'$ such that $\lambda'/\lambda$ and $\mu'/\mu^-$ are horizontal strips, and $\lambda^+/\lambda'$ and $\mu/\mu'$ are vertical strips. We weight such a selection with
$$(-1)^{|\mu/\mu^-|} (-q)^{|\lambda^+/\lambda'|+|\mu'/\mu^-|}.$$

We color the cells of $\lambda^+/\lambda'$ and $\mu'/\mu^-$ gray and leave the other cells white. So our requirements are saying that the gray cells of $\lambda^+/\lambda$ and the white cells of $\mu/\mu^-$ form a vertical strip, and the white cells of $\lambda^+/\lambda$ and the gray cells of $\mu/\mu^-$ form a horizontal strip; also, the white cells form a diagram of some shape $\lambda'/\mu'$ for $\lambda \subseteq \lambda' \subseteq \lambda^+$, $\mu^- \subseteq \mu' \subseteq \mu$. Furthermore, the weight of such an object is $(-1)^{|\mu/\mu^-|} (-q)^j$, where $j$ is the number of gray cells.

\begin{exm}
 Figure \ref{fig11} shows four examples with weights $q^{16}$, $q^{14}$, $q^{13}$ and $q^{11}$.
\begin{figure}[ht!]
 \begin{center}
  \includegraphics[height=3.5cm]{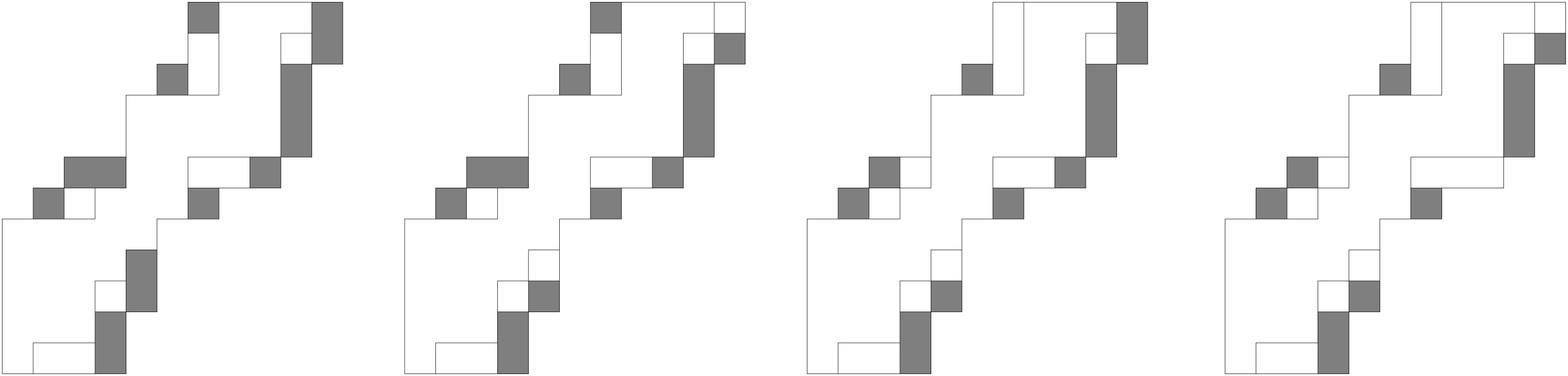}
 \end{center}
\caption{}\label{fig11}\end{figure}
\end{exm}

We claim that these objects indeed enumerate the right-hand side of SQMNR'.

\begin{lemma}
 For fixed $\lambda,\mu,\lambda^+,\mu^-$, we have
 $$\sum_{\lambda',\mu'} (-1)^{|\mu/\mu^-|} (-q)^{|\lambda^+/\lambda'|+|\mu'/\mu^-|} = (-1)^{|\mu/\mu^-|} (-q)^{\hgt(\lambda^+/\lambda)+\wt(\mu/\mu^-)} (1-q)^{\rib(\lambda^+/\lambda)+\rib(\mu/\mu^-)},$$
 where the sum on the left runs over all $\lambda',\mu'$ such that $\lambda'/\lambda$ and $\mu'/\mu^-$ are horizontal strips, and $\lambda^+/\lambda'$ and $\mu/\mu'$ are vertical strips.
\end{lemma}
\begin{proof}
 For each cell of $\lambda^+/\lambda$, we have to decide whether or not to put it in $\lambda'/\lambda$ or in $\lambda^+/\lambda'$ (i.e.\ whether to make it white or gray). If a cell in $\lambda^+/\lambda$ has a right neighbor in $\lambda^+/\lambda$, it cannot be in $\lambda^+/\lambda'$, since its right neighbor would also have to be in $\lambda^+/\lambda'$, and this would contradict the requirement that $\lambda^+/\lambda'$ is a vertical strip. Similarly, if a cell in $\lambda^+/\lambda$ has an upper neighbor in $\lambda^+/\lambda$, it cannot be in $\lambda'/\lambda$, since its upper neighbor would also have to be in $\lambda'/\lambda$, and this would contradict the requirement that $\lambda'/\lambda$ is a horizontal strip.\\
 This means that the colors of all the cells in $\lambda^+/\lambda$ are determined, except for the top right cell of each ribbon of $\lambda^+/\lambda$, which can be either white or gray.\\
 If a cell in $\mu/\mu^-$ has a right neighbor in $\mu/\mu^-$, it cannot be in $\mu/\mu'$, since its right neighbor would also have to be in $\mu/\mu'$, and this would contradict the requirement that $\mu/\mu'$ is a vertical strip. Similarly, if a cell in $\mu/\mu^-$ has an upper neighbor in $\mu/\mu^-$, it cannot be in $\mu'/\mu^-$, since its upper neighbor would also have to be in $\mu'/\mu^-$, and this would contradict the requirement that $\mu'/\mu^-$ is a horizontal strip.\\
 This means that the colors of all the cells in $\mu/\mu^-$ are determined, except for the top right cell of each ribbon of $\mu/\mu^-$, which can be either white or gray.\\
 In other words, we have two choices for each upper right cell of each ribbon of $(\lambda^+/\lambda) \cup (\mu/\mu^-)$. This already means that there are $2^{\rib(\lambda^+/\lambda)+\rib(\mu/\mu^-)}$ terms on the left-hand side.\\
 We have at least $\hgt(\lambda^+/\lambda)$ gray cells in $\lambda^+/\lambda$, and at least $\wt(\mu/\mu^-)$ gray cells in $\mu/\mu^-$. So the weight of a term on the left-hand side is
 $$(-1)^{|\mu/\mu^-|} (-q)^{\hgt(\lambda^+/\lambda)+\wt(\mu/\mu^-)} (-q)^j,$$
 where $j$ is the number of cells that are gray by choice, and these choices are made independently. Of course,
 $$\sum \binom{\rib(\lambda^+/\lambda)+\rib(\mu/\mu^-)}j (-1)^{|\mu/\mu^-|} (-q)^{\hgt(\lambda^+/\lambda)+\wt(\mu/\mu^-)} (-q)^j = $$
 $$= (-1)^{|\mu/\mu^-|} (-q)^{\hgt(\lambda^+/\lambda)+\wt(\mu/\mu^-)} (1-q)^{\rib(\lambda^+/\lambda)+\rib(\mu/\mu^-)},$$
 which finishes the proof of the lemma.
\end{proof}

We have managed to rewrite SQMNR'' as follows:
$$s_{\lambda/\mu} \cdot \left( \sum_{\tau \vdash r}
(1-q)^{\ell(\tau)} m_\tau\right) = \sum_{\lambda^+,\lambda',\mu^-,\mu'} (-1)^{|\mu/\mu^-|}(-q)^{|\lambda^+/\lambda'|+|\mu'/\mu^-|} s_{\lambda^+/\mu^-},$$
where the sum is over partitions $\lambda^+,\lambda',\mu^-,\mu'$ such that $\lambda^+/\lambda$ and $\mu/\mu^-$ are broken ribbons with $|\lambda^+/\lambda| + |\mu/\mu^-| = r$, $\lambda'/\lambda$ and $\mu'/\mu^-$ are horizontal strips, and $\lambda^+/\lambda'$ and $\mu/\mu'$ are vertical strips.

\medskip

For fixed $\lambda,\mu,r$, a term on the right-hand side of SQMNR'' therefore corresponds to a semistandard Young tableau $T$ with some cells colored white and some cells colored gray, such that the following properties are satisfied:
\begin{itemize}
 \item the shape of $T$ is $\lambda^+/\mu^-$ for some $\lambda^+ \supseteq \lambda$ and $\mu^- \subseteq \mu$, $|\lambda^+/\lambda| + |\mu/\mu^-| = r$, and $\lambda^+/\lambda$ and $\mu/\mu^-$ are broken ribbons;
 \item the white cells form a skew diagram $\lambda'/\mu'$ for some partitions $\lambda',\mu'$;
 \item the white cells in $\lambda^+/\lambda$ form a horizontal strip, and the white cells in $\mu/\mu^-$ form a vertical strip;
 \item the gray cells are in $(\lambda^+/\lambda) \cup (\mu/\mu^-)$, and they form a vertical strip in $\lambda^+/\lambda$ and a horizontal strip in $\mu/\mu^-$;
\end{itemize}
We call such an object a \emph{colored tableau} of shape $(\lambda,\mu,\lambda',\mu',\lambda^+,\mu^-)$. We weight a colored tableau by
$$(-1)^{|\mu/\mu^-|} (-q)^{|\lambda^+/\lambda'|+|\mu'/\mu^-|}.$$

Now perform the involution $\Psi$ on the gray cells of a colored tableau. More specifically, find $\Psi_{\lambda',\mu',\lambda^+,\mu^-}(T)$. Since $\lambda^+/\lambda'$ is a vertical strip and $\mu'/\mu^-$ is a horizontal strip, the map is well defined. One gray cell is removed, and one gray cell is added in the process. The result is a colored tableau $T'$ of shape $(\lambda,\mu,\lambda',\mu',\lambda^{++},\mu^{--})$ for some $\lambda^{++}$, $\mu^{--}$; it has the same white cells as $T$, the same number of gray cells as $T$, and with the property that $|\mu/\mu^{--}| = |\mu/\mu^{-}| \pm 1$ unless $T = T'$ is a fixed point.

\medskip

This already cancels a large number of terms. The ones that remain correspond to fixed points of $\Psi_{\lambda',\mu',\lambda^+,\mu^-}$. Each such fixed point consists of a a semistandard skew Young tableau $S$ of some shape $\lambda'/\mu'$, where $\lambda'/\lambda$ is a horizontal strip and $\mu/\mu'$ is a vertical strip, and of a strictly decreasing word $w$. Such an object is weighted by $(-1)^{|\mu/\mu'|} (-q)^{|w|}$.

\medskip

Now apply Assaf-McNamara involution $\Phi$ to the tableau. More specifically, find $\Phi_{\lambda,\mu,\lambda',\mu'}(S)$. The result is a semistandard Young tableaux $S'$ of some shape $\lambda^{+++}/\mu^{---}$ with the property that $|\mu/\mu^{---}| = |\mu/\mu^{-}| \pm 1$ unless $S = S'$ is a fixed point. This cancels more terms. The ones that remain correspond to fixed points of $\Phi_{\lambda,\mu,\lambda',\mu'}$, together with a strictly decreasing word $w$ and the weight $(-q)^{|w|}$. Each such fixed point consists of a semistandard Young tableau $R$ of shape $\lambda/\mu$, together with a weakly increasing word $v$ and a strictly decreasing word $w$. Such an object is weighted by $(-q)^{|w|}$. Furthermore, every such triple $(R,v,w)$ appears as a non-canceling term on the right. Indeed, insert the elements of $v$ into $R$ to get a semistandard Young tableau $S$ of shape $\lambda'/\mu$ for some partition $\lambda'$ so that $\lambda'/\lambda$ is a horizontal strip; then insert the elements of $w$ into $S$ and color the new cells gray to get a colored tableau $T$ of shape $\lambda^+/\mu$ for some partition $\lambda^+$ so that $\lambda^+/\lambda'$ is a vertical strip. Then applying $\Psi$ and $\Phi$ to $T$ yields $(R,v,w)$.

\medskip

It remains to enumerate all triples $(R,v,w)$. If we want $(v,w)$ to contain, say, $\tau_i$ copies of $i$, $1 \leq i \leq \ell$, we can choose any $j$-subset of $\{1,\ldots,\ell\}$ and put the elements in decreasing order to form $w$, and then put the remaining elements of the multiset $\{1^{\tau_1}, 2^{\tau_2}, \ldots,\ell^{\tau_\ell}\}$ in weakly increasing order to form $v$. Furthermore, the weight of $(R,v,w)$ for these $v$ and $w$ is $(-q)^j$. That means that the right-hand side of SQMNR'' becomes, after cancelations,
$$s_{\lambda/\mu} \cdot \left( \sum_{\tau \vdash r} (1-q)^{\ell(\tau)} m_\tau\right),$$
which is the left-hand side of SQMNR''.

\section{A proof via skew Littlewood-Richardson rule} \label{another}

It is informative to use Lam-Lauve-Sotille's \cite{lls} skew Littlewood-Richardson rule to find another proof of SQMNR. The first lemma is a simple computation that allows us to replace the quantum power sum functions with ``hook'' Schur functions and should remind the reader of the enumeration of pairs $(v,w)$ for $v$ a weakly increasing word,$v$ and strictly decreasing word $w$ at the end of the previous section. The second lemma is technical and states that a certain property is preserved in jeu de taquin slides. And the third lemma sheds some light on connections between jeu de taquin, hooks, and decompositions of broken ribbons into vertical and horizontal strips.

\begin{lemma} \label{lemma1}
 For all $r$, we have
 $$\wp_r =  \sum_{k = 1}^{r} (-q)^{r-k} s_{k,1^{r-k}}.$$
\end{lemma}
\begin{proof}
 Let us compute the expansion of the right-hand side in basis $m_\lambda$. Given $\lambda = (\lambda_1,\lambda_2,\ldots,\lambda_\ell)$ and $k$, $1 \leq k \leq r$, it is easy to count the number of semistandard Young tableaux of shape $(k,1^{r-k})$ and type $\lambda$: place $1$ in cell $(1,1)$, choose the elements to place in the (only) cell of rows $2,\ldots,r-k+1$ in $\binom {\ell-1}{r-k}$ ways and place them in the first column in strictly increasing order, and place the remaining elements in weakly increasing order in the first row. This tells us that the coefficient of $m_\lambda$ in the right-hand side is
 $$\sum_{k = 1}^{r} (-q)^{r-k} \binom {\ell-1} {r-k} = (1-q)^{\ell-1},$$
 which is also the coefficient of $m_\lambda$ in $\wp_r$.
\end{proof}

For the second lemma, we have to recall the celebrated \emph{backward (respectively, forward) jeu de taquin slide} due to Sch\" utzenberger. Say we are given a skew standard Young tableau of shape $\lambda/\mu$. Let $c = c_0$ be a cell that is not in $\lambda/\mu$, shares the right or lower edge (respectively, the left or upper edge) with $\lambda/\mu$, and such that $\lambda/\mu \cup c$ is a valid skew diagram. Let $c_1$ be the cell of $\lambda/\mu$ that shares an edge with $c_0$; if there are two such cells, take the one with the smaller entry (respectively, larger entry). Then move the entry occupying $c_1$ to $c_0$, look at the tableau entries below or to the right of $b_1$ (respectively, above or to the left of $b_1$), and repeat the same procedure. We continue until we reach the boundary, say in $m$ moves. The new tableau is a standard Young tableau and is called $\jdt_c(T)$. We say that $c_0, c_1,\ldots,c_m$ is the \emph{path} of the slide.

\medskip

If $T$ is a skew standard Young tableau, we can repeatedly perform backward jeu de taquin slides. The final result $S$ is a standard Young tableau of straight shape, and it is independent of the choices during the execution of the algorithm. We say that $T$ \emph{rectifies} to $S$. See \cite[Appendix A1.2]{stanley}.

\medskip

Say we are given a standard Young tableau $T$ of shape $\lambda/\mu$. We say that $T$ has the \emph{$k$-NE property}1 if the following statements are true:
\begin{itemize}
 \item[NE1] the entry in the last cell of the first non-empty row (i.e.\ the northeast cell) of $\lambda/\mu$ is $k$;
 \item[NE2] if $i,j < k$, then $i$ appears strictly to the left of $j$ in $T$;
 \item[NE3] if $i,j > k$, then $i$ appears strictly above $j$ in $T$.
\end{itemize}

\medskip

The following figure shows some tableaux with $4$-NE property.

\begin{figure}[ht!]
 \begin{center}
  \includegraphics[height=2.5cm]{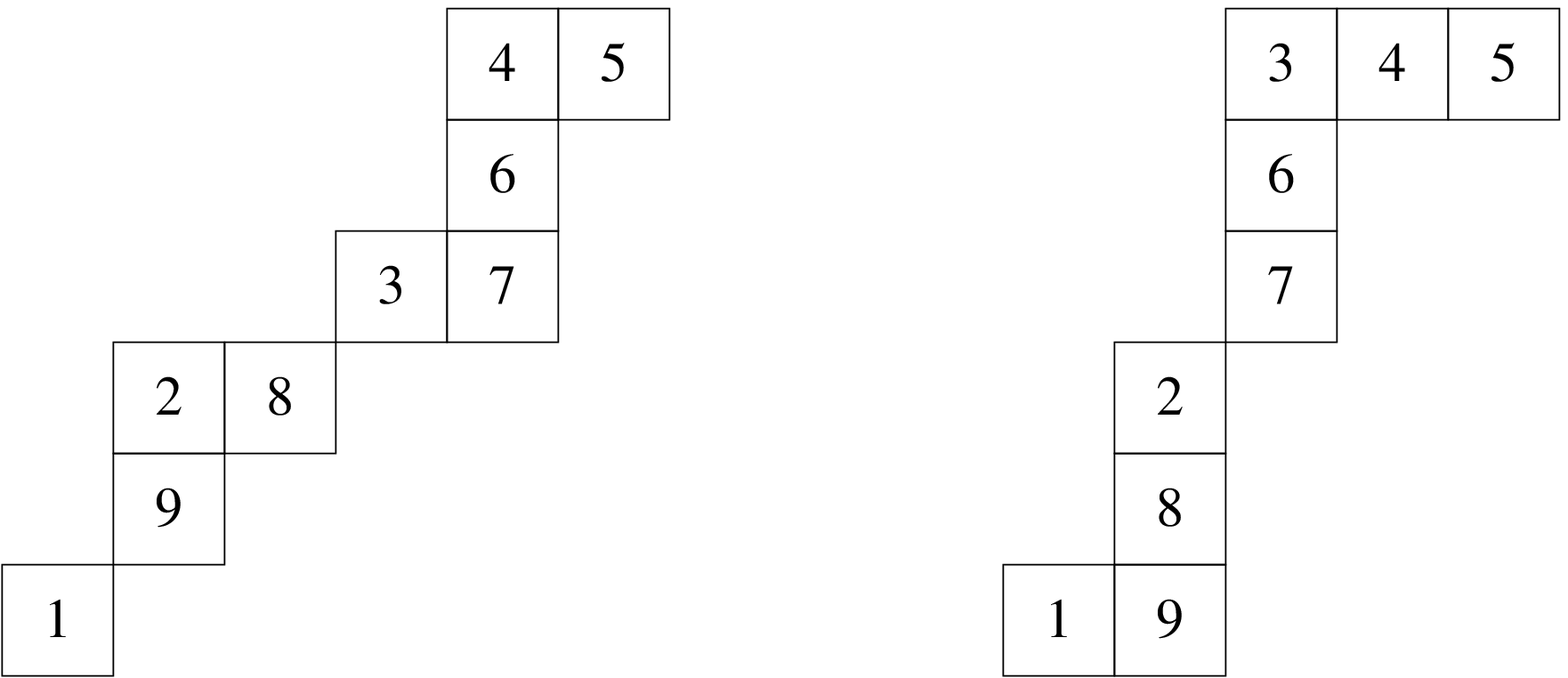}
 \end{center}
\caption{}\label{fig16}\end{figure}

\begin{lemma} \label{lemma3}
 If a tableau $T$ has the $k$-NE property, its shape is a broken ribbon. Furthermore, the $k$-NE property is preserved in a jeu de taquin slide.
\end{lemma}
\begin{proof}
 For the first statement, assume that there is a $2 \times 2$ square in $\lambda/\mu$, assume it has numbers $a,b$ in the upper row and $c,d$ in the lower row. If $c < k$, then $a < c$ implies $a < k$, and this is a contradiction with property NE2. If $c > k$, then $d > c$ implies $d > k$, and this is contradiction with property NE3. But we cannot have $c = k$ since $c$ is obviously not the northeast cell of $T$.\\
 Now take the path $c_0, c_1,\ldots,c_m$ of a backward slide in $T$. We claim that either all $c_i$ are in the same row, or all $c_i$ are in the same column, or $m = 2$, $c_1$ is below $c_0$, and $c_2$ is to the right of $c_1$.\\
 If they are not in the same row or column, it means that the $c_i$'s take a turn. All three squares involved in the term (i.e.\ $c_i,c_{i+1},c_{i+2}$, where either $c_{i+1}$ lies to the right of $c_i$ and $c_{i+2}$ lies below $c_{i+1}$, or $c_{i+1}$ lies below $c_i$ and $c_{i+2}$ lies to the right of $c_{i+1}$) cannot be in $\lambda/\mu$, since that would imply that there is a $2 \times 2$ block in $\lambda/\mu$. Therefore the only option is if the entries involved in the turn are $c_0,c_1,c_2$. Say that $c_1$ lies to the right of $c_0$ and $c_2$ lies below $c_1$. The fact that we can add $c_0$ to $\lambda/\mu$ implies that there is a cell $c'$ of $\lambda/\mu$ below $c_0$, with entry, say, $a_3$. Say that we have $a_1$ in $c_1$ and $a_2$ in $c_2$. We have $a_3 < k$ (since $a_3 > k$ would imply $a_2 > a_3 > k$, and this would contradict NE3). We also have $a_1 < a_3$, since otherwise we would be sliding from $c'$ into $c_0$ rather than from $c_1$. But then $a_1 < k$ lies to the right of $a_3 < k$, even though $a_1 < a_3$, which contradicts NE2.\\
 That means that $c_1$ lies below $c_0$ and $c_2$ lies to the right of $c_1$. If $m \geq 3$, there must be cells both above $c_2$ and $c_3$ in $\lambda/\mu$, and this would give a $2 \times 2$ block.\\
 Assume first that all $c_i$ are in the same row. Obviously the northest cell is preserved, so property NE1 holds for the new tableau. Futhermore, since all cells of the tableau stay in the same row, NE3 is preserved. If NE2 is violated in the new tableau, it must mean that there is a cell $c'$ with entry $< k$ in $T$ in the same column as $c_0$. But this can only happen if this cell lies immediately below $c_0$; by NE2, its entry is less than the entry of $c_1$, and therefore we would slide from $c'$ into $c_0$, not from $c_1$.\\
 If all $c_i$ are in the same column, the proof that the properties NE1, NE2 and NE3 are preserved is completely analogous. So let us assume that we have $m = 2$, $c_1$ is below $c_0$, and $c_2$ is to the right of $c_1$. There must be a cell $c'$ of $\lambda/\mu$ to the right of $c_0$, say with entry $a_3$. Assume we have $a_1$ in $c_1$ and $a_2$ in $c_2$. Then $a_1 < k$ ($a_1 > k$ would imply $a_2 > k$ and contradict NE3) and $a_2 > k$ ($a_2 < k$ would imply $a_3 < k$ and contradict NE2). So all cells with entries $< k$ stay in the same column, and all cells with entries $>k$ stay in the same row. Therefore NE2 and NE3 are still satisfied, and it is clear that the northeast cell stays in place.\\
 The proof for a forward slide is analogous. This completes the proof of the lemma.
 \end{proof}

\begin{lemma} \label{lemma2}
 Take $r,k$, $1 \leq k \leq r$, and let $S$ be the standard Young tableau of shape $(k,1^{r-k})$ with $1,2,\ldots,k$ in the first row, and $k+1,k+2,\ldots,r$ in rows $2,3,\ldots,r-k+1$. Choose a skew shape $\lambda/\mu$. Then the number of standard Young tableaux of shape $\lambda/\mu$ that rectify to $S$ is $\binom{\rib(\lambda/\mu) - 1} {k - 1 - \wt(\lambda/\mu)}$  if $\lambda/\mu$ is a broken ribbon of size $r$, and $0$ otherwise.
\end{lemma}
\begin{proof}
 Obviously the number is $0$ unless $|\lambda/\mu| = r$.\\
 Note that $S$ has the $k$-NE property. By Lemma \ref{lemma3}, that means that if $T$ of shape $\lambda/\mu$ rectifies to $S$, $T$ has the $k$-NE property and its shape $\lambda/\mu$ is a broken ribbon. Furthermore, there is only one non-skew standard Young tableau that has the $k$-NE property, and that is $S$.\\
 It remains to assume that $\lambda/\mu$ is a broken ribbon of size $r$, and to count the number of standard Young tableaux of shape $\lambda/\mu$ that have the $k$-NE property. Place $k$ in the northeast cell. If a cell in $\lambda/\mu$ has a right neighbor in $\lambda/\mu$, then the entry has to be less than $k$ (otherwise both this entry and the entry to the right would be greater than $k$, and this would contradict NE3). Similarly, if a cell in $\lambda/\mu$ has an upper neighbor in $\lambda/\mu$, then the entry has to be greater than $k$ (otherwise both this entry and the entry above it would be less than $k$, and this would contradict NE2).\\
 This means that there are at least $\wt(\lambda/\mu)$ elements that are $< k$. We can choose the northeast element of any ribbon except the northeast ribbon and make it $< k$. Since there are $k-1$ elements total that are less than $k$, we have
 $$\binom{\rib(\lambda/\mu) - 1} {k - 1 - \wt(\lambda/\mu)}$$
 choices.
\end{proof}

Finally, recall the following result from \cite{lls}. For standard Young tableaux $T$ and $S$, we let $T * S$ be the tableau we get by placing $T$ below and to the left of $S$. See Figure \ref{fig14} for an example.

\begin{figure}[ht!]
 \begin{center}
  \includegraphics[height=3.5cm]{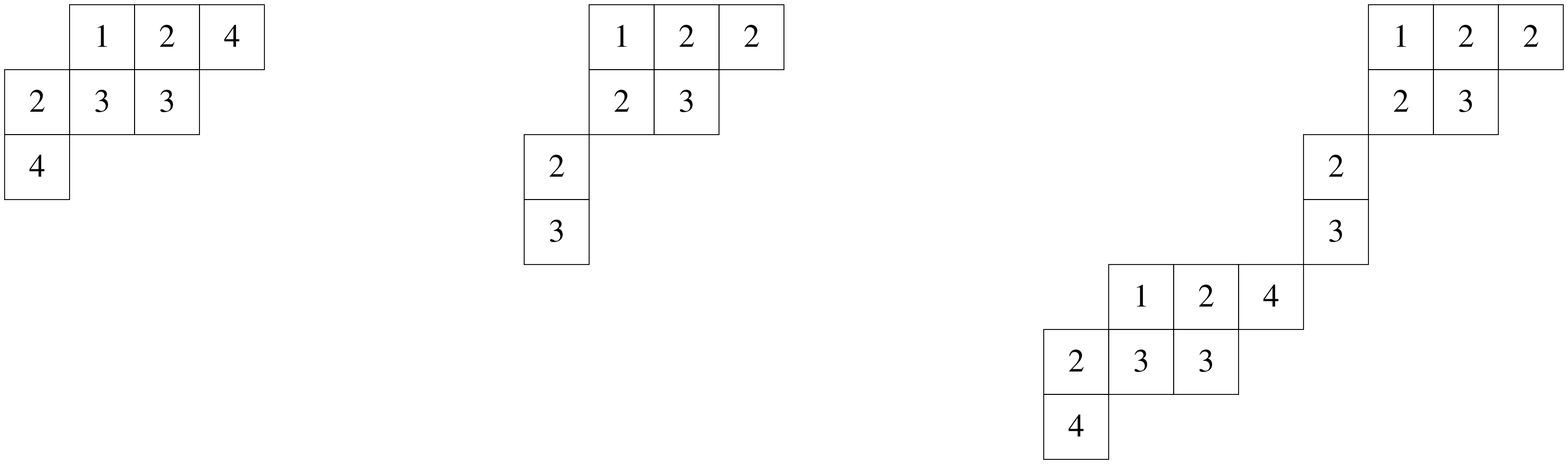}
 \end{center}
\caption{Tableaux $T$, $S$ and $T*S$.}\label{fig14}\end{figure}

\begin{thm}[Skew Littlewood-Richardson rule -- SLRR]
 Let $\lambda$, $\mu$, $\sigma$, $\tau$ be partitions and fix a tableau
 $T$ of shape $\sigma$. Then
 $$s_{\lambda/\mu} s_{\sigma/\tau} = \sum (-1)^{|R^-|} s_{\lambda^+/\mu^-},$$
 where the sum is over triples $(R^-, R^+, R)$ of standrd Young tableaux of respective shapes $(\mu/\mu^-)^c$, $\lambda^+/\lambda$ and $\tau$ such that $R^- * R^+ * R$ rectifies to $T$.
\end{thm}

The lemmas indeed prove SQMNR as follows. By Lemma \ref{lemma1},
$$s_{\lambda/\mu} \cdot \wp_r = \sum_{k = 1}^{r} (-q)^{r-k} s_{\lambda/\mu} \cdot s_{k,1^{r-k}}.$$
By SLRR,
$$s_{\lambda/\mu} \cdot s_{k,1^{r-k}} = \sum_{R^-,R^+} (-1)^{|R^-|} s_{\lambda^+/\mu^-},$$
where the sum is over $R^- \in \syt((\mu/\mu^-)^c)$, $R^+ \in \syt(\lambda^+/\lambda)$ such that $R^- * R^+$ rectifies to $T$, where $T$ is the standard Young tableau of shape $(k,1^{r-k})$ with $1,2,\ldots,k$ in the first row, and $k+1,k+2,\ldots,r$ in rows $2,3,\ldots,r-k+1$. By Lemma \ref{lemma2}, the sum on the right is over $\lambda^+,\mu^-$ such that $\lambda^+/\lambda$ and $\mu/\mu^-$ are broken ribbons, and for such $\lambda^+,\mu^-$, the coeffiecient of $s_{\lambda^+/\mu^-}$ is
$$(-1)^{|\mu/\mu^-|} \binom{\rib(\lambda^+/\lambda) + \rib(\mu/\mu^-) - 1} {k - 1 - \wt(\lambda^+/\lambda) - \hgt(\mu/\mu^-)}.$$
This means that the coefficient of $s_{\lambda^+/\mu^-}$ in $s_{\lambda/\mu} \cdot \wp_r$ is
$$(-1)^{|\mu/\mu^-|} \sum_k (-q)^{r-k}  \binom{\rib(\lambda^+/\lambda) + \rib(\mu/\mu^-) - 1} {k - 1 - \wt(\lambda^+/\lambda) - \hgt(\mu/\mu^-)}.$$
Since $r = \rib(\lambda^+/\lambda) + \rib(\mu/\mu^-) + \wt(\lambda^+/\lambda) + \wt(\mu/\mu^-) + \hgt(\lambda^+/\lambda) + \hgt(\mu/\mu^-)$, the sum equals
$$(-q)^{\hgt(\lambda^+/\lambda) + \wt(\mu/\mu^-)} \sum_k (-q)^{\rib(\lambda^+/\lambda) + \rib(\mu/\mu^-) - 1 - (k - 1 - \wt(\lambda^+/\lambda) - \hgt(\mu/\mu^-))} \textstyle \binom{\rib(\lambda^+/\lambda) + \rib(\mu/\mu^-) - 1} {k - 1 - \wt(\lambda^+/\lambda) - \hgt(\mu/\mu^-)}$$
$$ = (-q)^{\hgt(\lambda^+/\lambda) + \wt(\mu/\mu^-)} (1-q)^{\rib(\lambda^+/\lambda) + \rib(\mu/\mu^-) - 1}$$
by the binomial theorem. This is SQMNR'.

\section{Some conjectures involving Hall-Littlewood polynomials} \label{hlp}

The quantum power sum functions $\wp_r$ are equal to Hall-Littlewood polynomials $P_r$ (with parameter $q$ instead of the usual $t$), see e.g.\ \cite[page 214]{mcd}. So while SPR gives the expansion of $s_{\lambda/\mu} s_r$, SQMNR  gives the expansion of $s_{\lambda/\mu} P_r$. Of course, the expansion of $P_\lambda P_r$ and $P_\lambda P_{1^r} = P_\lambda e_r$ in terms of $P_{\lambda^+}$ are two of the basic results for Hall-Littlewood polynomials (see \cite[\S III, (3.2) and (3.10)]{mcd}). The following questions naturally arise. Can we exchange the roles of $P$ and $s$ in SQMNR, i.e.\ is there a natural expansion of $P_{\lambda/\mu} s_r$ in terms of $P_{\lambda^+/\mu^-}$? What about $P_{\lambda/\mu} e_r$? And can we find a skew version of the Pieri rule for Hall-Littlewood polynomials, an expansion of $P_{\lambda/\mu} P_r$? The following conjectures suggest that the answers to all these questions are in the affirmative.

\medskip

Recall the definition of the $q$-binomial coefficient,
$$\begin{bmatrix} n \\ k \end{bmatrix}_q = \frac{[n]!}{[k]![n-k]!}, \mbox{ where } [i]! = 1(1+q)(1+q+q^2) \cdots (1+q+q^2+\ldots+q^{i-1}).$$

\medskip

For a horizontal strip $\lambda/\mu$, define 
$$\hs(\lambda/\mu) = \prod_{\begin{array}{c} {\scriptstyle \lambda^c_i = \mu^c_i+1} \\ {\scriptstyle \lambda^c_{i+1} = \mu^c_{i+1}} \end{array}} (1-q^m_i(\lambda)).$$
For a vertical strip $\lambda/\mu$, define
$$\vs(\lambda/\mu) = \prod_{i \geq 1} \begin{bmatrix} \lambda^c_i - \lambda^c_{i+1} \\ \lambda^c_i - \mu^c_i \end{bmatrix}_q.$$
For a broken ribbon $\lambda/\mu$, define 
$$\br(\lambda/\mu) = (-q)^{\hgt(\lambda/\mu)} (1-q)^{\rib(\lambda^+/\lambda)}.$$
For any skew shape $\lambda/\mu$, define
$$\sk(\lambda/\mu) = q^{\sum_i \binom{(\lambda^+)^c_i - \lambda^c_i}{2}} \prod_i \begin{bmatrix}  (\lambda^+)^c_i - \lambda^c_{i+1} \\ m_i(\lambda) \end{bmatrix}_q.$$

With this notation, SQMNR' can be expressed as
$$s_{\lambda/\mu} \cdot P_r = \frac 1{1-q} \sum_{\lambda^+,\mu^-} (-1)^{|\mu/\mu^-|} \br(\lambda^+/\lambda) \br((\mu/\mu)^c) s_{\lambda^+/\mu^-},$$
where the sum on the right is over $\lambda^+,\mu^-$ such that $\lambda^+/\lambda$ and $\mu/\mu^-$ are broken ribbons with $|\lambda^+/\lambda| + |\mu/\mu^-| = r$.

\begin{conj}
 For partitions $\lambda,\mu$, $\mu \subseteq \lambda$, and $r \geq 0$ we have
 $$P_{\lambda/\mu} \cdot s_r = \sum (-1)^{|\mu/\mu^-|} \sk(\lambda^+/\lambda) P_{\lambda^+/\mu^-},$$
 where the sum on the right is over all $\lambda^+ \supseteq \lambda$, $\mu^- \subseteq \mu$ such that $\mu/\mu^-$ is a vertical strip and $|\lambda^+/\lambda|+|\mu/\mu^-| = r$.
\end{conj}

For $\lambda = \mu = \emptyset$, this is identity (2) on page 219 in \cite{mcd}.

\begin{conj}
 For partitions $\lambda,\mu$, $\mu \subseteq \lambda$, and $r \geq 0$ we have
 $$P_{\lambda/\mu} \cdot e_r = P_{\lambda/\mu} \cdot P_{1^r} = \sum (-1)^{|\mu/\mu^-|} \vs(\lambda^+/\lambda) P_{\lambda^+/\mu^-},$$
 where the sum on the right is over all $\lambda^+ \supseteq \lambda$, $\mu^- \subseteq \mu$ such that $\lambda^+/\lambda$ and $(\mu/\mu^-)^c$ are vertical strips and $|\lambda^+/\lambda|+|\mu/\mu^-| = r$.
\end{conj}

For $\mu = \emptyset$, this is \cite[\S III, (3.2)]{mcd}

\begin{conj}
 For partitions $\lambda,\mu$, $\mu \subseteq \lambda$, and $r \geq 0$ we have
 $$P_{\lambda/\mu} \cdot P_r = \frac 1{1-q} \sum (-1)^{|\mu/\mu^-|} \hs(\lambda^+/\lambda) \br((\mu/\mu^-)^c) P_{\lambda^+/\mu^-},$$
 where the sum on the right is over all $\lambda^+ \supseteq \lambda$, $\mu^- \subseteq \mu$ such that $\lambda^+/\lambda$ is a horizontal strip, $\mu/\mu^-$ is a broken border strip and $|\lambda^+/\lambda|+|\mu/\mu^-| = r$.
\end{conj}

For $\mu = \emptyset$, this is \cite[\S III, (3.10)]{mcd}.

\medskip

The methods of this paper do \emph{not} seem to work for these three conjectures. In other words, the sign-reversing involutions described in Sections \ref{amcn} and \ref{inv} cancels only the constant coefficients on both sides of conjectured equalities; positive powers of $q$ cancel in some other, mysterious manner.

\section{Final remarks} \label{final}

\subsection{} The motivation for this work was the open problem posed by Assaf and McNamara in \cite{slides}: to find a combinatorial proof of the skew Murnaghan-Nakayama rule (SMNR). Even though this paper provides a completely bijective proof of the skew \emph{quantum} Murnaghan-Nakayama rule, which obviously specializes to the non-quantum rule, Assaf-McNamara's problem remains open. Indeed, plugging $q = 1$ into SQMNR'', which is the identity we proved bijectively, gives $0$ on both sides. To get SMNR, we have to divide SQMNR'' by $1-q$ and then set $q = 1$.

\medskip

One possibility seems to be to instead find a bijective proof of SQMNR'. This would mean that one of the northeast corners of ribbons of $(\lambda^+/\lambda) \cup (\mu/\mu^-)$ would have to be colored white (or gray), perhaps the northeast corner of $\lambda^+/\lambda$ or the northeast corner of $\mu/\mu^-$. We were unable to find such a bijection. Even such a bijection, however, would not be enough to construct a bijection that proves SMNR. Indeed, plugging in $q = 1$ makes many of the skew tableaux weighted with $0$, and hence would not appear on the right-hand side of SMNR at all. We would want to avoid such $0$-weight objects in the sign-reversing involution.

\medskip

One possibility seems to to construct an involution-principle type of a bijection. Namely, given a skew semistandard Young tableau of shape $\lambda^+/\mu$, with $\lambda^+/\lambda$ a ribbon of size $r$, we would map it to a tableau of shape $\lambda^{++}/\mu^{--}$, where $\lambda^{++}/\lambda$ is a broken ribbon of size $r-1$ and $|\mu/\mu^{--}|$ a broken ribbon of size $1$, then map this to a tableau of shape $\lambda^{+++}/\mu^{---}$, where $\lambda^{+++}/\lambda$ is a broken ribbon of size $r-2$ and $|\mu/\mu^{--}|$ a broken ribbon of size $2$ etc., until finally reaching a tableau of shape $\mu/\mu^-$, where $\mu/\mu^-$ is a ribbon of size $r$. We leave this as motivation for further work.

\medskip

\subsection{} There is another natural $q$-version of power sum functions, defined by
$$\bp_r = \sum_{\tau \vdash r} q^{r-\ell(\tau)}(q-1)^{\ell(\tau)-1} m_\tau,$$
$$\bp_\mu = \bp_{\mu_1} \bp_{\mu_2} \cdots.$$
For example,
$$\bp_4 = q^3 m_4 + q^2(q-1) m_{31} + q^2(q-1) m_{22} + q(q-1)^2 m_{211} + (q-1)^3 m_{1111}$$
and
$$\bp_{22} = q^2 m_4 + 2q(q-1)m_{31} + (3q^2-2q+1) m_{22} + 2(q-1)(2q-1) m_{211} + 6 (q-1)^2 m_{1111}.$$

We have
$$p_r|_{q = 0} = (-1)^{r-1} m_{1^r}, \qquad \bp_r|_{q = 1} = m_r = p_r, \qquad \lim_{q \to \infty} \frac{\bp_r}{q^{r-1}} = \sum_{\tau \vdash r} m_\tau = s_r$$

\begin{thm}[SQMNR''']
 For partitions $\lambda,\mu$, $\mu \subseteq \lambda$, and $r \geq 0$, we have
 $$s_{\lambda/\mu} \cdot \bp_r = (-1)^{r-1} \sum_{\lambda^+,\mu^-} (-1)^{|\mu/\mu^-|}(-q)^{\wt(\lambda^+/\lambda)+\hgt(\mu/\mu^-)} (1-q)^{\rib(\lambda^+/\lambda)+\rib(\mu/\mu^-)-1} s_{\lambda^+/\mu^-},$$
 where the sum on the right is over $\lambda^+,\mu^-$ such that $\lambda^+/\lambda$ and $\mu/\mu^-$ are broken ribbons with $|\lambda^+/\lambda| + |\mu/\mu^-| = r$.
\end{thm}

For $q = 0$, this is the conjugate skew Pieri rule (multiplied by $(-1)^{r-1}$), for $q = 1$, this is again the skew Murnaghan-Nakayama rule, and if we divide by $q^{r-1}$ and send $q$ to $\infty$, we get the skew Pieri rule.

\medskip

\subsection{} Lam-Lauve-Sotille's skew Littlewood-Richardson rule is very general, but the computation of actual coefficients in the expansion, i.e.\ counting all standard Young tableaux of a given shape that rectify to a given tableau, is complicated in practice. In light of Section \ref{another}, our work can be seen as one possible answer to the following question. For what special shapes of $\lambda,\mu, \sigma,\tau$ can we actually compute the coefficients? SQMNR can be interpreted as saying that if $\tau = \empty$ and $\sigma$ is a hook, the coefficients are certain binomial coefficients, while SPR says that the coefficient is $\pm 1$ if $\tau = \empty$ and $\sigma = r$.

\medskip

It would be interesting to find other examples when the coefficients can be computed and yield elegant answers, both for Schur functions and for other Hopf algebras.

\end{document}